\documentclass[12pt]{amsart}
\usepackage{amssymb}
\usepackage{enumerate,color}
\usepackage{graphicx}
\usepackage[text={6in,9in},centering]{geometry}
\usepackage{subfigure}
\usepackage{epstopdf}
\usepackage[colorlinks,linkcolor=blue,anchorcolor=blue,citecolor=blue]{hyperref}

\theoremstyle{plain}
  \newtheorem{thm}{Theorem}[section]
  \newtheorem{cor}[thm]{Corollary}
  \newtheorem{lem}[thm]{Lemma}
  \newtheorem{prop}[thm]{Proposition}
\theoremstyle{definition}
  
  \newtheorem{exmp}{Example}
\theoremstyle{remark}

\def\R{\mathbb{R}}

\def\D{\mathcal{D}}

\def\card{\textrm{card}}

\def\I{{\rm 1 \hskip-2.9truept l}}   

\newcommand{\ba}{\mathbf{a}}
\newcommand{\proj}{\mathrm{proj}}
\newcommand{\diag}{\mathrm{diag}}

\newcommand{\dimh}{\dim_{\rm H}}
\newcommand{\dimb}{\dim_{\rm B}}
\newcommand{\bu}{\mathbf{u}}
\newcommand{\bv}{\mathbf{v}}

\newcommand{\w}{\omega} 
\newcommand{\be}{\begin{equation}} 
\newcommand{\ee}{\end{equation}} 

\newcommand{\f}{\infty}

\begin{document}
\title[Dimension theory of Non-Autonomous iterated function systems]{Dimension theory of Non-Autonomous iterated function systems}
\author{Yifei Gu}
\address{Department of Mathematics, East China Normal University, No. 500, Dongchuan Road, Shanghai 200241, P. R. China}

\email{52275500012@stu.ecnu.edu.cn}

\author{Jun Jie Miao}
\address{Department of Mathematics, East China Normal University, No. 500, Dongchuan Road, Shanghai 200241, P. R. China}

\email{jjmiao@math.ecnu.edu.cn}

\begin{abstract}
In the paper, we define a class of  new fractals named ``non-autonomous attractors", which are the generalization of classic Moran sets and attractors of iterated function systems. Simply to say, we replace the similarity mappings by contractive mappings and remove the separation assumption in Moran structure. We give the dimension estimate for non-autonomous attractors.

Furthermore, we study a class of non-autonomous  attractors, named `` non-autonomous affine sets or affine sets'', where the contractions are restricted to affine mappings. To study the dimension theory of such fractals,  we define two critical values $s^*$ and $s_A$, and the upper box-counting dimensions and Hausdorff dimensions of non-autonomous affine sets are bounded above by $s^*$ and $s_A$, respectively. Unlike self-affine fractals where $s^*=s_A$, we always have that $s^*\geq s_A$, and the inequality may strictly hold.

Under certain conditions, we obtain that the upper box-counting dimensions and Hausdorff dimensions of non-autonomous affine sets may equal to $s^*$ and $s_A$, respectively. In particular, we  study non-autonomous affine sets with random translations, and  the Hausdorff dimensions of such sets equal to $s_A$ almost surely.

\end{abstract}

\maketitle

\section{Introduction}

\subsection{Self-affine sets}
Let $\{\Psi_{i}\}_{1\leq i\leq M}$ be a finite set of affine contractions on
$\R^{d}$ with $M\geq 2$  and
\begin{equation}
    \Psi_{i}(x) = T_{i}(x) +a_{i} \qquad (1\leq i\leq M), \label{ifs}
\end{equation}
where $a_{i}$ is a translation vector, and $T_i$ is a linear transformation on $\R^{d}$. The set $\{\Psi_{i}\}_{1\leq i\leq M}$ is called a {\em self-affine  iterated function system}.  By the well-known theorem of
Hutchinson, see~\cite{Bk_KJF2,Hutch81}, the IFS has a unique
{\em attractor}, that is a unique non-empty compact $E \subset \R^{d}$
such that
\begin{equation}\label{saattractor}
E = \bigcup_{1\leq i\leq M} \Psi_{i}(E),
\end{equation}
which is called  a {\em self-affine set}. If the affine transformations $\Psi_i$ are similarity mappings, we call $E$ a {\em self- similar set}, see~\cite{Bk_KJF2} for details.

Formulae giving the Hausdorff and box-counting dimensions of self-similar sets satisfying the open set condition are well-known, see~\cite{Bk_KJF2}. However, calculation of the dimensions
of self-affine sets is more awkward, see~\cite{Bedfo84,Falco13, FH08,HYH94,McMul84,PerSo00}. For
each $k=1,2,\cdots$, we write
$\mathcal{I}^{k}=\{i_{1}i_{2}\cdots i_{k}:1\leq i_{j}\leq
M,\ j\leq k\}$
for the set of words of length $k$. Let $\phi^s$ be the singular value function defined by ~\eqref{def svf}.
Falconer in  \cite{Falco88} defined the criticla value $d(T_1,\ldots,T_M)$ which is  the unique solution to
$$
\lim_{k\to\infty}\left(\sum_{\mathbf{i}\in\mathcal{I}^{k}} \phi^s(T_\mathbf{i})\right)^{\frac{1}{k}}=1,
$$
and it is often called  \textit{affine dimension} or \textit{Falconer dimension}.
Given  $\|T_i\|<\frac{1}{2}$ for $i=1,2,\ldots M,$ it turns out that
\begin{equation}\label{sadimf}
\dimh F = \dimb F =\min \{d, d(T_1, \ldots, T_M)\},
\end{equation}
for almost all $\ba=(a_1,\ldots,a_M)\in\R^{Md}$(in the sense of $Md$-dimensional Lebesgue measure), we refer the readers to~\cite{Falco88, Solom98} for details. Note that $d(T_1, \ldots, T_M)$ is always the upper bound for the box-counting dimension of self-affine sets,  Falconer in \cite{Falco92} proved that the box-counting dimension and affine dimension coincide by applying projection condition with other assumptions.  From then on, a considerable amount of literature has been published on the validation of formula~\eqref{sadimf} under various conditions. In particular,  in~\cite{JPS07}, Jordan, Pollicott and Simon studied perturbed self-affine sets where they changed translations into independently identically distributed random variables, and they showed that for  $\|T_i\|<1$, $i=1,2,\ldots M,$ the dimension formula~\eqref{sadimf} holds almost surely.  In~\cite{FK18}, Falconer and Kempton showed the dimension formula~\eqref{sadimf} holds for all translations in $R^2$ under various assumptions. We refer readers to~\cite{BHR19,Falco92,Falco05,Falco10, Feng23, FJJ, HocRap, HYH98, HueLal95, MorShm} for various related studies.

\subsection{Moran sets}
Moran sets were first studied by Moran in~\cite{Moran}, and we recall the definition for the readers’ convenience.

Let $\{n_{k}\}_{k\geq 1}$ be a sequence of integers greater than or equal to $2$. For
each $k=1,2,\cdots$, we write
\begin{equation}\label{defsigmak}
\Sigma^{k}=\{u_{1}u_{2}\cdots u_{k}:1\leq u_{j}\leq
n_{j},\ j\leq k\}
\end{equation}
for the set of words of length $k$, with $
\Sigma^{0}=\{\emptyset \}$ containing only the empty word $\emptyset$, and write
\begin{equation}\label{defsigma*}
\Sigma^{*}=\bigcup _{k=0}^{\infty }\Sigma^{k}
\end{equation}
for the set of all finite words.

Suppose that $
J\subset \mathbb{R}^{d}$ is a compact set with $\mbox{int}(J)\neq
\varnothing $ (we always write int($\cdot )$ for the interior of a set). Let $\{\Xi_{k}\}_{k\geq 1}$ be a sequence of positive real vectors where $\Xi_{k}=(c_{k,1},c_{k,2},\cdots, c_{k,n_{k}})$ for each $k\in \mathbb{N}$. We say the
collection $\mathcal{J}=\{J_{\mathbf{u}}:\mathbf{u}\in \Sigma^*\}$ of closed
subsets of $J$ fulfils the \textit{Moran structure} if it satisfies the
following Moran structure conditions (MSC):
\begin{itemize}
  \item[(1).] For each $\mathbf{u}\in \Sigma^*$, $J_{\mathbf{u}}$ is geometrically similar to
$J$, i.e., there exists a similarity $\Psi_{\mathbf{u}}:\mathbb{R}%
^{d}\rightarrow \mathbb{R}^{d}$ such that $J_{\mathbf{u}}=\Psi_{\mathbf{u}}(J)$%
. We write $J_{\emptyset }=J$ for the empty word $\emptyset $.
  \item[(2).] For all $k\in \mathbb{N}$ and $\mathbf{u}\in \Sigma^{k-1}$, the elements $J_{%
\mathbf{u}1}, J_{\mathbf{u}2},\cdots, J_{\mathbf{u}n_{k}}$ of $\mathcal{J}$
are the subsets of $J_{\mathbf{u}}$ with disjoint interiors, ie., $\mbox{int}%
(J_{\mathbf{u}i})\cap \mbox{int}(J_{\mathbf{u}i^{\prime }})=\varnothing $
for $i\neq i^{\prime }$. Moreover, for all $1\leq i\leq n_{k} $,
\begin{equation*}
\frac{|J_{\mathbf{u}i}|}{|J_{\mathbf{u}}|}=c_{k,i},
\end{equation*}%
where $|\cdot |$ denotes the diameter.
\end{itemize}
The non-empty compact set
$$
E=E(\mathcal{J})=\bigcap\nolimits_{k=1}^{\infty }\bigcup\nolimits_{%
\mathbf{u}\in \Sigma^{k}}J_{\mathbf{u}}
$$
is called a \textit{Moran set} determined by $%
\mathcal{J}$. For each $\mathbf{u}\in \Sigma^{k}$, the element $J_{\mathbf{u}}$ is  called a \textit{\ $k$th-level basic set} of $E$.
For each integer $k> 0,$ let $d_k$ be the unique real solution
of the equation
\begin{equation}  \label{eqns}
\prod\nolimits_{i=1}^{k}\left(
\sum\nolimits_{j=1}^{n_{i}}(c_{i,j})^{d_k}\right)=1.
\end{equation}
Let $d_{\ast }$ and $d^{\ast }$ be the real numbers given respectively  by
\begin{equation}\label{defd}
d_{\ast }=\liminf_{k\rightarrow \infty }d_k, \quad
d^{\ast }=\limsup_{k\rightarrow \infty }d_k.
\end{equation}
It was shown in \cite{HL96, Wen00,Wen01} that if $c_*=\inf \{c_{k,j} :  k\in \mathbb{N}, 1\leq j\leq n_k  \}>0,$ the following dimension formulae hold
\begin{equation*}
\dimh E=d_{\ast }, \quad \dim _{\rm P}E={\overline{\dim}}_{\rm B} E=d^{\ast }.
\end{equation*}

The dimension theory of Moran sets has been studied extensively, and we refer the readers to~\cite{HL96, Wen00, Wen01} for detail and references therein. Note that, in the definition of  Moran sets, the position of $J_{\mathbf{u}i}$ in $J_{\mathbf{u}}$ is very flexible, and the contraction ratios may also vary at each level. Therefore the structures of Moran sets are more complex than self-similar sets,  and in general, the inequality
$$
\dimh E\leq \underline{\dim} _{\rm B}E \leq {\overline{\dim}}_{\rm B} E
$$
holds strictly for Moran fractals.
The general lower box dimension formula for Moran sets is still an open question. Except providing various examples, Moran sets are also useful tools for analysing properties of fractal sets in various studies, for example, see~\cite{Wen00} and references therein for applications.

Note that similarities and separation assumption are required in the Moran structure,  Inspired by the structure of iterated function systems,  we generalize the Moran structure to non-autonomous structure where we replace similarities by contractions or  affine mappings, and remove the separation assumption.  The geometric properties of non-autonomous sets are more complex than Moran sets since mappings may have different contraction ratios in different directions. Therefore, non-autonomous sets provide not only interesting phenomena but also useful tools for fractal analysis. Different to the attractors of iterated function systems, non-autonomous sets do not have dynamical properties any more, and the tools in ergodic theory cannot be invoked. As a result, it is more difficult to explore their dimension formulae and other fractal properties.

In this paper, we investigate the dimension theory of non-autonomous attractors, and particularly, we are interested in a classs of attractors, named ``non-autonomous affine sets or affine sets''.  In section \ref{sec 2}, we first give the definition of  non-autonomous attractors and estimate the dimensions of such sets, then we study the non-autonomous affine sets and state the main conclusions in subsection \ref{subMC}.  The dimension estimates for non-autonomous sets are provided in section~\ref{sec_NAS}. Upper and lower box dimensions of non-autonomous affine sets are explored in section~\ref{secPB}. We discuss the upper bounds of Hausdorff dimensions of non-autonomous affine sets in section~\ref{secH}, and we give the  Hausdorff dimension formula for a spectial class of non-autonomous affine sets with finitely many translations. In section~\ref{secASA}, we study the Hausdorff dimension of non-autonomous affine sets with random translations, and we prove that with probability one, the Hausdorff dimensions of such sets equal to a critical value. Finally, in section \ref{sec_CV}, we compare  critical values of non-autonomous affine sets, and we also provide some examples to illustrate our conclusions.

\section{Non-autonomous iterated function systems}\label{sec 2}

\subsection{Definition of non-autonomous iterated function systems}
Let $\{n_{k}\}_{k=1}^\infty$ be a sequence of integers such that $n_k\geq 2$ for all $k\geq 1$. Let $\Sigma^{k}$ and $\Sigma^{*}$ be given by \eqref{defsigmak} and \eqref{defsigma*}, respectively. Suppose that $J\subset \R^d$ is a compact set with non-empty interior.
Let $\{\Xi_k\}_{k=1}^\infty$ be a sequence of collections of contractive mappings, that is
\begin{equation}\label{Xi}
\Xi_k=\{S_{k,1},S_{k,2},\ldots,S_{k,n_k}\},
\end{equation}
where each $S_{k,j}$ satisfies that
$|S_{k,j}(x)-S_{k,j}(y)| \leq c_{k,j} |x-y| $ for some $0<c_{k,j}<1$.
We say the collection $\mathcal{J}=\{J_{\mathbf{u}}:\mathbf{u}\in \Sigma^*\}$ of closed subsets of $J$ fulfils the \textit{non-autonomous structure with respect to $\{\Xi_k\}_{k=1}^\infty$} if it satisfies the following conditions:
\begin{itemize}
\item[(1).] For all integers $k>0$ and all $\mathbf{u}\in \Sigma^{k-1}$, the elements $J_{\mathbf{u}1}, J_{\mathbf{u}2},\cdots, J_{\mathbf{u}n_{k}}$ of $\mathcal{J}$ are the subsets of $J_{\mathbf{u}}$. We write $J_{\emptyset }=J$ for the empty word $\emptyset $.
\item[(2).] For each $\mathbf{u}=u_1\ldots u_k \in \Sigma^*$,  there exists an  transformation $\Psi_\mathbf{u}: \mathbb{R}^{d}\rightarrow \mathbb{R}^{d}$  such that
    $$
    J_{\mathbf{u}}=\Psi_{\mathbf{u}}(J)=\Psi_{u_1}\circ \ldots \circ \Psi_{u_j} \ldots \circ \Psi_{u_k} (J),
    $$
    where $\Psi_{u_j}(x)=S_{j,u_j}x+ \w_{u_1\ldots u_j}$, for some $\w_{u_1\ldots u_j} \in \R^d$, and $S_{j,u_j}\in \Xi_j$, $j=1,2,\ldots k$.
\item[(3).]The maximum of the diameters of $J_\bu$ tends to 0 as $|\bu|$ tends to $\infty$, that is,
    $$
    \lim_{k\to \infty} \max_{\bu\in \Sigma^k} |J_\bu|=0.
    $$

\end{itemize}
We call $\{\Xi_k\}_{k=1}^\infty$ \textit{ a non-autonomous iterated function systems(NIFS)} determined by $\mathcal{J}$, and the non-empty compact set
\begin{equation}\label{attractor}
E=E(\mathcal{J})=\bigcap\nolimits_{k=1}^{\infty }\bigcup\nolimits_{\mathbf{u}\in \Sigma^{k}}J_{\mathbf{u}}
\end{equation}
is called a \textit{non-autonomous attractor} determined by $\mathcal{J}$.  For all $\mathbf{u}\in \Sigma^{k}$, the elements $J_{\mathbf{u}}$ are called \textit{\ $k$th-level basic sets} of $E$.  If the non-autonomous attractor $E$ satisfies  that for all integers $k\geq 1$ and $\mathbf{u}\in \Sigma^{k-1}$,
$$\mathrm{int} (J_{\bu i})\cap \mathrm{int} (J_{\bu i'}) = \emptyset  \quad \textit{ for } i\neq i'\in\{1,2,\ldots, n_k\},
$$
we say $E$ satisfies \textit{Moran separation condition} (MSC).

\vspace{0.2cm}
\noindent\textit{Remark} (1) Overlap is allowed in the  structure, and Moran sets are special non-autonomous attractors satisfying Moran separation condition. Note that  Moran sets are always uncountable, but the cardinality of non-autonomous attractors sets may be finite, countable or uncountable even if  Moran separation condition is satisfied, see Example \ref{SMfnt} and Example \ref{SMcntb} in section \ref{sec_CV}.

(2) In the definition, the transformation $\Psi_\bu$ may be determined in the following way.
    For each $\bu=u_1\in \Sigma^1$ and $J_\bu\in \mathcal{J}$, there exists an mapping $\Psi_{\bu}:J\to J_\bu$ such that
      $$
      J_{\bu}=\Psi_{\bu}(J)=S_{1,u_1}(J)+\w_{u_1} \qquad\qquad \textit{ for some } \w_{u_1} \in \R^d.
      $$
      Suppose that for each $\bu\in \Sigma^{k-1}$  and $J_\bu\in \mathcal{J}$, the affine mapping $\Psi_{\bu}:J\to J_\bu$ is defined. For each $1\leq j\leq n_k$ and $J_{\bu j}\in \mathcal{J}$, there exists an affine mapping $\Psi_{\bu j}:J\to J_{\bu j}$ such that
      $$
      J_{\bu j}=\Psi_{\bu j}(J)=\Psi_{\mathbf{u}}\circ \Big(S_{k,j}(J)+\w_{\bu j} \Big) \qquad \textit{ for some } \w_{\bu j} \in \R^d.
      $$

(3) The assumption $\lim_{k\to \infty} \max_{\bu\in \Sigma^k} |J_\bu| = 0$ in the definition is necessary, otherwise the set $E$ may have positive finite Lebesgue measure, and we do not consider such case in this paper, see Example ~\ref{exp1} in section \ref{sec_CV}.

(4) Non-autonomous iterated function system is  also a generalization of the iterated function system. However, for $\bu\neq\bv\in\Sigma^{k-1}$,  the transformations $S_{k,j}(x)+\w_{\bu j} $ and $S_{k,j}(x)+\w_{\bv j} $ in $\Psi_{\bu j}$ and $\Psi_{\bv j}$ have the same  contractive part $S_{k,j}\in \Xi_k$ but with different translations.  Therefore  the non-autonomous attractor $E$ may  be different to the attractor of classic iterated function system even if the sequence $\{\Xi_k\}$ is identical, that is, $n_k =M$ and $\Xi_k=\{S_1, \ldots,S_M\}$ for all $k>0$.

For general non-autonomous sets, it is difficult to find their fractal dimensions, but  we are still able to  provide some rough estimates if contraction ratios are known.
\begin{thm}\label{dimub}
Let $E$ be the non-autonomous attractor given by~\eqref{attractor}. For each integer $ k\geq 1$, we assume that
$$
|S_{k,i}(x)-S_{k,i}(y)| \leq c_{k,i}|x-y|, \qquad (x,y)\in J
$$
where $c_{k,i}<1$, for all $i=1,2,\ldots, n_k$.  Let $d_*$ and $d^*$ be given by ~\eqref{defd}. Then
$$\dimh E \leq \min\{d_*, d\}.$$
Furthermore, if $c_*=\inf \{c_{k,j} :  k\in \mathbb{N}, 1\leq j\leq n_k  \}>0$, we have that
$$
\overline{\dim}_{\rm B} E \leq \min\{ d^*, d\}.
$$
\end{thm}

We next obtain a lower bound for Hausdorff and box-counting dimensions in the case where  the basic sets of $E$ satisfy the following condition.
We say the non-autonomous set $E$ satisfies {\it gap separation condition(GSC)} if there exists a constant $C$ such that for all $\bu\in\Sigma^*$, $i\neq j$, we have that
$$
\inf\{|x-y|: x\in J_{\bu i}, y\in J_{\bu j }\}\geq C|J_\bu|
$$

\begin{thm}\label{dimlb}
Let $E$ be the non-autonomous attractor given by~\eqref{attractor} with gap separation condition satisfied. For each integer $ k\geq 1$, we assume that
$$
|S_{k,i}(x)-S_{k,i}(y)| \geq c_{k,i}|x-y|, \qquad (x,y)\in J
$$
where $c_{k,i}<1$,  for all $i=1,2,\ldots, n_k$. Let $d_*$ and $d^*$ be given by ~\eqref{defd}.Then
$$
\dimh E \geq d_* ,\qquad  \qquad\overline{\dim}_{\rm B} E \geq d^*.
$$
\end{thm}
Note that if the contractions are all similarities in the NIFS $\{\Xi_k\}_{k=1}^\infty$, then $c_*>0$ is frequently used to  guarranttee the lower bounds in Theorem~\ref{dimlb}. For general NIFSs,  $c_*>0$ is not sufficient for the theorem, see Example~\ref{SMcntb} for a counterexample.

To study the dimension properties of non-autonomous sets, the following notations are frequently used in our context.
Let $\Sigma^{k}$ and $\Sigma^{*}$ be given by \eqref{defsigmak} and \eqref{defsigma*}, respectively. Let $\Sigma^\infty= \{(u_1u_2 \ldots u_k\ldots) : 1 \leq  u_k \leq n_k\}$ be the corresponding set of infinite words, where $\{n_{k}\geq 2\}_{k\geq 1}$ is the sequence of integers.

For $\bu=u_1\cdots u_k\in\Sigma^k$, we write $\bu^-=u_1\ldots u_{k-1}$ and  write $|\bu|=k$ for the length of $\bu$. For each $\bu = u_1 u_2\cdots u_k\in
\Sigma^{*}$, and $\bv = v_1 v_2\cdots \in \Sigma^{\infty}$,   we say $\mathbf{u}$ is a \textit{curtailment} of $\bv$, denote by  $\mathbf{u} \preceq \bv$, if $\mathbf{u} = v_1\cdots v_k =\bv|k$. We call the set $\mathcal{C}_\bu =
\{\bv\in\Sigma^{\infty} : \bu \preceq\bv\}$ the \textit{cylinder} of $\mathbf{u}$, where $\bu\in \Sigma^*$. If $\mathbf{u}=\emptyset$, its cylinder is $\mathcal{C}_\bu=\Sigma^{\infty}$.

For $\bu, \bv\in \Sigma^\infty$, let $\bu\wedge \bv\in \Sigma^*$ denote the maximal common initial finite word of both $\bu$ and $\bv$.
We topologise $\Sigma^\infty$  using the metric
$d(\bu,\bv) = 2^{-|\bu \wedge \bv |}$ for distinct $\bu,\bv \in \Sigma^\infty$ to make $\Sigma^\infty$  a compact metric space.
The cylinders
$\mathcal{C}_\bu = \{\bv \in \Sigma^\infty : \bu \preceq \bv \}$ for  $\bu \in \Sigma^*$ form a base
of open and closed neighbourhoods for $\Sigma^\infty$.
We call a set of finite words $A \subset \Sigma^*$ a \textit{covering set} for $\Sigma^\infty$ if $\Sigma^\infty \subset \bigcup_{\mathbf{u}\in A}\mathcal{C}_\bu$.

For $\bv=v_1\ldots v_k\in \Sigma^k$, we denote compositions of mappings by $S_\mathbf{v}= S_{1,v_1}\ldots S_{k,v_k}$. Let $\Pi: \Sigma^\infty \rightarrow \mathbb{R}^d$ be the projection given by
\begin{eqnarray}
\Pi(\bu) & =& \bigcap_{k=0}^\infty (S_{1,u_1} + \w_{u_1})  (S_{2,u_2} + \w_{u_1u_2}) \cdots (S_{k,u_k} + \w_{\bu|k})(J)  \nonumber\\
& =& \lim_{k \to \infty} (S_{1,u_1} + \w_{u_1})  (S_{2,u_2} + \w_{u_1u_2}) \cdots (S_{k,u_k} + \w_{\bu|k})(x)\label{points2}\\
& =&  \w_{u_1} +S_{1,u_1}\w_{u_1u_2}+ \cdots +  S_{\bu|k}\w_{\bu|k+1}  + \cdots.  \nonumber
\end{eqnarray}
It is clear that the attractor $E$ is the image of $\Pi$, i.e. $E=\Pi(\Sigma^\infty)$.  Note that the projection $\Pi$ is surjective. To emphasize the dependence on translations, we sometimes write $\Pi^\w(\bu)$ and $E^\w$ instead of $\Pi(\bu)$ and $E$.

Let $\mu$  be a finite Borel regular measure on $\Sigma^\infty$. We define  $\mu^\w$, the projection of the  measure $\mu$ onto $\R^d$, by
\begin{equation}
 \mu^\w(A)=\mu\{\bu: \Pi^\w(\bu)\in A\},                    \label{def_ua}
\end{equation}
for $A\subseteq\mathbb{R}^d$, or equivalently by
\be \int
f(x)d\mu^\w(x)=\int f(\Pi^\w(\bu))d\mu(\bu),  \label{def_ua2} \ee
for every continuous  $f:\R^d\to \mathbb{R}$. Then $\mu^\w$ is a Borel
measure supported by $E^{\w}$.

\subsection{Non-autonomous affine iterated function systems}
Let $\{\Xi_k\}_{k=1}^\infty$ be a sequence of collections of contractive matrices, that is
\begin{equation}\label{Xi}
\Xi_k=\{T_{k,1},T_{k,2},\ldots,T_{k,n_k}\},
\end{equation}
where
$T_{k,j}$ are $d\times d$ matrices with $\|T_{k,j}\|<1$ for $j=1,2,\ldots n_k$. The collection $\mathcal{J}=\{J_{\mathbf{u}}:\mathbf{u}\in \Sigma^*\}$ of closed subsets of $J$ fulfils the non-autonomous structure with respect to the sequence $\{\Xi_k\}_{k=1}^\infty$ of contractive matrices. We call $\{\Xi_k\}_{k=1}^\infty$  the \textit{non-autonomous affine iterated function system (NAIFS)} determined by $\mathcal{J}$, and we  call the attractor $E$ given by \eqref{attractor}
a \textit{non-autonomous affine set(NAS) or affine set} determined by $\mathcal{J}$.

Note that Moran sets may be regarded as a special case of non-autonomous affine sets satisfying Moran separation condition, where $T_{k,i}$ is the multiplication of contraction ratio $c_{k,i}$ and  a rotation matrix $A_{k,i}$.


Let $T\in\mathcal{L}(\R^d,\R^d)$ be a contracting and non-singular linear mapping. The {\em singular values} $\alpha_1(T),\alpha_2(T),\ldots$, $\alpha_d(T)$ of $T$ are the lengths of the (mutually perpendicular) principle semi-axes of $T(B)$, where $B$ is the unit ball in $\R^d$. Equivalently they are the positive square roots of the eigenvalues of $T^\ast T$, where
$T^\ast$ is the transpose of $T$. Conventionally, we write that  $1>\alpha_1(T)\geq \ldots\geq \alpha_d(T)>0$.

For $0\leq s \leq d$, the \textit{singular value function} of $T$ is defined by
\begin{equation}\label{def svf}
\phi^s(T)=\alpha_1(T)\alpha_2(T)\ldots \alpha_{m-1}(T) \alpha_m^{s-m+1}(T),
\end{equation}
where $m$ is the integer such that $m-1<s\leq m$. For technical convenience, we set $\phi^s(T)= (\alpha_1(T)\alpha_2(T)\ldots \alpha_d(T))^{s/d}=(\det T)^{s/d}$ for $s > d$. It is clear that $\phi^s(T)$ is continuous and strictly decreasing in $s$. The singular value function is submultiplicative, that is, for all $s\geq 0$,
\begin{equation}\label{subphi}
\phi^s(TU)\leq \phi^s(T)\phi^s(U),
\end{equation}
for all $T,U\in \mathcal{L}(\R^d,\R^d)$, see ~\cite{Falco88} for details.

Let $\{\Xi_k\}_{k=1}^\infty$ be given by \eqref{Xi}. We write
\begin{eqnarray}\label{alpha}
\alpha_+&=&\sup\{\alpha_1(T_{k,i}) : 1\leq i\leq n_k  \textit{,  } k \in \mathbb{N}\}, \\
\alpha_-&=&\inf\{\alpha_d(T_{k,i}) : 1\leq i\leq n_k  \textit{,  } k \in \mathbb{N}\}. \nonumber
\end{eqnarray}
Immediately, for each $\bu\in \Sigma^k$, the singular value function $\phi^s$ of $T_\bu$ is bounded by
\begin{equation}\label{alhpha+-}
\alpha_-^{sk}\leq \phi^s(T_\bu) \leq \alpha_+^{sk}.
\end{equation}

Note that for  a self-affine set $E$, the sequence $\{\sum_{\bu\in\mathcal{I}^{k}}\phi^s(T_\bu)\}_{k=1}^\infty$ is submultiplicative, and this implies that the function
$$
p(s)=\lim_{k\to\infty}\left(\sum_{\bu\in\mathcal{I}^{k}} \phi^s(T_\bu)\right)^{\frac{1}{k}}
$$
is continuous and strictly decreasing in $s$.
This property plays an important role in finding the affine dimensions of self-affine fractals. However, in general, the submultiplicative property does not hold for non-autonomous affine sets. The lack of submultiplicativity causes one of the main difficulties to determine the dimensions of non-autonomous affine sets.

\subsection{Main conclusions for non-autonomous affine sets}\label{subMC}
Given $\{n_{k}\geq 2\}_{k=1}^\infty$. Let $\{\Xi_k\}_{k=1}^\infty$ be the NAIFS~\eqref{Xi} and  $E$ be the corresponding non-autonomous affine set defined by\eqref{attractor}.             From now on, we always assume that the matrices in $\Xi_k$ for all $ k \in \mathbb{N}$ are nonsingular, and
\begin{equation} \label{aspalpha}
0< \alpha_- \leq   \alpha_+<1.
\end{equation}

For each $s>0$ and $0<\epsilon <1$, let $m$ be the integer such that $m-1<s\leq m$ and define
$$
\Sigma^*(s, \epsilon)=\{\mathbf{u}=u_1\ldots u_k \in \Sigma^*: \alpha_m(T_\mathbf{u})  \leq\epsilon < \alpha_m(T_{\bu^-}) \}.
$$
The set $\Sigma^*(s,\epsilon)$ is a \textit{cut-set} or \textit{stopping} in the sense  that for every $\mathbf{u}\in \Sigma^\infty$, there is a unique integer $k$ such that $ \mathbf{u}|k\in \Sigma^*(s, \epsilon)$. For each $\mathbf{u}\in \Sigma^*(s,\epsilon)$, by \eqref{subphi} and \eqref{alpha},  we have that
$$
\alpha_-\epsilon < \alpha_m(T_\mathbf{u}) < \epsilon.
$$
We define
\begin{equation}\label{defs*}
s^*=\inf\Big\{s: \limsup_{\epsilon\to 0}\sum_{\bu\in\Sigma^*(s, \epsilon)}\phi^s(T_\bu)<\infty\Big\}.
\end{equation}

The critical value $s^*$ plays a key role in the box dimensions of non-autonomous affine sets. The following theorem shows that $s^*$ is an upper bound for the upper box dimension of $E$.

\begin{thm}\label{dimBU}
Let $E$ be the non-autonomous affine set given by ~\eqref{attractor}. Then
$$
\overline{\dim}_{\rm B} E \leq \min\{s^*,d\}.
$$
\end{thm}

Suppose that the matrices in $\Xi_k$ for all $k>0$ are scalar matrices, that is $T_{k,j}=\diag\{\alpha_{k,j},\ldots,\alpha_{k,j}\}$ is a $d\times d$ scalar matrix for each $1\leq j\leq n_k$ and each $k>0$.  Suppose that $E$ satisfies the MSC. Then $E$ is a Moran set, and $s^*$ gives the upper box dimension.
\begin{cor}\label{cor1}
Suppose that the matrices $T\in \Xi_k$ are scalar matrices for all $k>0$.
Let $E$ be the non-autonomous affine set given by ~\eqref{attractor} and satisfying MSC.  Then
$$
\overline{\dim}_{\rm B} E = s^*=d^*,
$$
where $d^*$ is given by \eqref{defd}.
\end{cor}

Next we show that under certain strong restrictions, the critical value $s^*$ gives the upper box dimension of $E$.
We say the non-autonomous affine set $E$ satisfies the \textit{open projection condition}(OPC) if  there exists an open set $U$ such that $J\subset \overline{U}$ and for each $k>0$ and $\bu \in \Sigma^{k-1}$,
$$
U\supset \bigcup_{i=1,2,\ldots, n_k} (T_{k,i}+\w_{\bu,i})(U),
$$
with the union disjoint, and
 $$\mathcal{L}^{d-1}\{\proj_\Theta U\}=\mathcal{L}^{d-1}\{\proj_\Theta \overline{U}\},
 $$
 for all $(d-1)$-dimensional subspaces $\Theta$.

\begin{thm}\label{projc}
Let  $E$ be the non-autonomous affine set  given by ~\eqref{attractor} and satisfying OPC. Suppose that  there exists  $c>0$ such that
$$
\mathcal{L}^{d-1}\{\proj_\Theta (J\cap \Psi_\bu^{-1}(E))\}\geq c,
$$
for all $(d-1)$-dimensional subspaces $\Theta$ and all $\bu\in \Sigma^*$. Then
$$
\overline{\dim}_{\rm B} E = s^*.
$$

\end{thm}

For self-affine fractals, the box-counting dimensions and Hausdorff dimensions are bounded by the same value $d(T_1,\ldots, T_M)$. Unfortunately, $s^*$ does not provide much information for the Hausdorff dimensions of non-autonomous affine sets, and we have to find a different candidate for Hausdorff dimensions.

Fix $s\geq 0$. By applying ``Method II'' in Rogers~\cite{Bk_Rogers},
we define a Hausdorff type measure on $\Sigma^\infty$ as follows. For each integer $k>0$, let
$$
\mathcal{M}_{(k)}^s(G)=\inf \Big\{\sum_{\bu} \phi^s(T_\mathbf{u}): G\subset \bigcup_\mathbf{u} \mathcal{C}_\bu, |\mathbf{u}|\geq k \Big\}.
$$
We obtain a net measure of Hausdorff type by letting
\begin{equation}\label{measure}
\mathcal{M}^s(G)=\lim_{k\to \infty} \mathcal{M}_{(k)}^s(G),
\end{equation}
for all $G\subset \Sigma^\infty$.
Note that $\mathcal{M}^s$ is an outer measure which restricts to a measure on the Borel subsets of $\Sigma^\infty$.

We define
\begin{equation}\label{defsA}
s_A=\inf\{s:\mathcal{M}^s(\Sigma^\infty)=0\}= \sup\{s:\mathcal{M}^s(\Sigma^\infty)=\infty\}.
\end{equation}
The critical value $s_A$ is important in studying the Hausdorff dimensions for non-autonomous affine sets.  The following theorem shows that $s_A$ is an upper bound for the Hausdorff dimension of $E$.
\begin{thm}\label{dimHU}
Let $E$ be the non-autonomous affine set given by ~\eqref{attractor}. Then
$$
\dimh E \leq \min\{ s_A, d\}.
$$
\end{thm}
In section \ref{sec_CV}, we discuss the relations of $s^*$, $s_A$ and $d(T_1,\ldots, T_M)$, and we also give some examples to show that $s_A$ and $s^*$ are sharp bounds for Hausdorff dimensions and upper box dimensions of non-autonomous affine sets, respectively, see Example~\ref{exCVie} in section \ref{sec_CV}.

Given the discontinuity of dimensions of $E^\w$ in the translations $\w$, we can only expect to show that $s_A$ is also a lower bound for Hausdorff dimensions for almost all constructions, in some sense.

First, we consider a special case where the translations of affine mappings in the non-autonomous structure are selected only from a finite set.
Let $\Gamma=\{a_1,\ldots, a_\tau\}$ be a finite collection of translations, where $a_1,\ldots, a_\tau$ are regarded later as variables in $\R^d$. For each $\mathbf{u}=u_1\ldots u_k \in \Sigma^*$,  we have that $J_{\mathbf{u}}=\Psi_{\mathbf{u}}(J)=\Psi_{u_1}\circ  \ldots \circ \Psi_{u_k} (J)$. Suppose that the translation of $\Psi_{u_j}$ is an element of $\Gamma$, that is,
 $$
\Psi_{u_j}(x)=T_{j,u_j}x+ \w_{u_1\ldots u_j}, \qquad \w_{u_1\ldots u_j}\in \Gamma,
  $$
 for $j=1,2,\ldots, k$.
  We write $\ba=(a_1,\ldots, a_\tau)$ as a variable in $\R^{\tau d}$. To emphasize the dependence on these special translations in $\Gamma$,  we denote the non-autonomous affine set by $E^\ba$.  The following conclusion shows that the Hausdorff dimension of $E^\ba$ equals $s_A$ almost surely.
\begin{thm}\label{dimHL}
Given $\Gamma=\{a_1,\ldots, a_\tau\}$. Let $E^\ba$ be the non-autonomous affine set given by ~\eqref{attractor} where the translations of affine mappings are chosen from $\Gamma$.  Suppose that
$$
\sup\{\|T_{k,j}\| : 0<j\leq n_k,  k>0 \}< \frac{1}{2}.
$$
Then for $\mathcal{L}^{\tau d}$-almost all $\ba\in \R^{\tau d}$,

$(1)$ $\dimh E^\ba = s_A$ if $s_A \leq d$,

$(2)$ $\mathcal{L}^d(E^\ba)>0$ if $s_A>d$.
\end{thm}

As you can see the set $E^\ba$ is special and unnatural, the reason is that there is no obvious candidate which would take the place of the Lebesgue measure in infinite dimensional spaces. Inspired by~\cite{Falco10,JPS07}, we study the Hausdorff dimensions of non-autonomous affine sets in probabilistic language.

Let $\mathcal{D}$ be a bounded region in $\R^d$. For each $\bu\in \Sigma^* $, let $\omega_\bu\in  \mathcal{D} $ be a random vector distributed
according to some Borel probability measure $P_\bu$ that is absolutely continuous with respect to $d$-dimensional Lebesgue measure. We assume that the $\omega_\bu$ are independent identically
distributed random vectors. Let $\mathbf{P} $ denote the product probability measure $\mathbf{P} = \prod_{\bu\in \Sigma^*} P_\bu$ on
the family $\omega =\{\omega_\bu : \bu \in  \Sigma^* \}.$
In this context, for each $\mathbf{u}=u_1\ldots u_k \in \Sigma^*$,  we assume that the translation of $\Psi_{u_j}$ is an element of $\w$, that is,
 $$
\Psi_{u_j}(x)=T_{j,u_j}x+ \w_{u_1\ldots u_j}, \qquad \w_{u_1\ldots u_j}\in \omega =\{\omega_\bu : \bu \in  \Sigma^* \},
  $$
 for $j=1,2,\ldots, k$. We also assume that the collection of  $J_{\mathbf{u}}=\Psi_{\mathbf{u}}(J)=\Psi_{u_1}\circ  \ldots \circ \Psi_{u_k} (J)$ fulfils the non-autonomous structure, and we call $E^\omega$ the {\it non-autonomous affine set with random translations}.

Next theorem states that, in this probabilistic setting, the Hausdorff dimension of $E^\omega$ equals $s_A$ with probability one.

\begin{thm}\label{asams}
Let $E^\w$ be the non-autonomous affine set with random translation. Then for $\mathbf{P}$-almost all $\omega$,

$(1)$ $\dimh E^\w = s_A$ if $s_A \leq d$,

$(2)$ $\mathcal{L}^d(E^\omega)>0$ if $s_A>d$.

\end{thm}

Finally, we explore the lower box-counting dimension of non-autonomous affine sets. Under open projection condition, we show that $s_A$ is a lower bound for the lower box dimension.
\begin{thm}\label{dimLBL}
Let $E$ be the non-autonomous affine set given by ~\eqref{attractor} with the open projection condition satisfied. Suppose that  there exists  $c>0$ such that
$$\mathcal{L}^{d-1}\{\proj_\Theta (J\cap \Psi_\bu^{-1}(E))\}\geq c,
$$ for all $(d-1)$-dimensional subspaces $\Theta$ and all $\bu\in \Sigma^*$.
Then
$$\underline{\dim}_{\rm B} E \geq s_A.$$
\end{thm}

\begin{cor}\label{cor2}
Let $E$ be  the non-autonomous affine set  in $\R^2$ given by ~\eqref{attractor} with  the open projection condition satisfied. Suppose that $E$ has a connected component which is not contained in any straight line, and  $J\cap \Psi_\bu^{-1}(E)=E$ for each $\bu\in\Sigma$.
Then
$$\overline{\dim}_{\rm B} E=s^*,\qquad \underline{\dim}_{\rm B} E \geq s_A.$$
\end{cor}

\section{Dimension estimates of non-autonomous sets}\label{sec_NAS}

In this section, we estimate the dimensions of the attractor $E$ of an NIFS consisting of contractions which are no similarities.

First, we show that the Hausdorff dimension and upper  box-counting dimension  are upper bounded by $d_*$ and $d^*$ given by ~\eqref{defd}, respectively, if
for all $ k\geq 1$,
$$
 |S_{k,i}(x)-S_{k,i}(y)| \leq c_{k,i}|x-y|, \qquad (x,y)\in J
$$
where $c_{k,i}<1$,  for all $i=1,2,\ldots, n_k$.

\begin{proof}[Proof of Theorem~\ref{dimub}]
For $\bu=u_1\ldots u_k \in \Sigma^*$, we write $u^*=u_1\ldots u_{k-1}$ and  $c_\bu=c_{1,u_1}\ldots c_{k,u_k}$. Since for all $ k\geq 1$,
$$
|S_{k,i}(x)-S_{k,i}(y)| \leq c_{k,i}|x-y|, \qquad (x,y)\in J
$$
for every $i=1,2,\ldots, n_k$, it is clear that $|J_\bu| \leq c_\bu |J|$.

First, we prove that $\dimh E \leq d_*$.
For $t>d_*$, there exits a sequence $\{k_i\}$ such that $t>d_{k_i}$, and it follows that $\prod_{l=1}^{k_i}\left(\sum_{j=1}^{n_l} (c_{l,j})^t\right)<1$. For $\delta>0$, there exists $i$ such that $\max_{\bu \in \Sigma^{k_i}} |J_\bu| < \delta$, and the set $\{J_\bu: \bu \in \Sigma^{k_i}\}$ is a $\delta$-cover of $E$. Hence
\begin{eqnarray*}
\mathcal{H}_\delta^t(E) &\leq& \sum_{\bu\in \Sigma^{k_i}}(c_{1,u_1}\ldots c_{k_i,u_{k_i}} |J|)^t \\
&=& \prod_{l=1}^{k_i}\left(\sum_{j=1}^{n_l} (c_{l,j})^t\right) |J|^t \\
&<& |J|^t.
\end{eqnarray*}
By taking $\delta\to 0$, we have $\mathcal{H}^t(E)  \leq |J|^t < \infty$. Since $t$ is arbitrarily chosen, it follows that $\dimh E \leq d_*$.

Next, suppose that  $c_* = \inf_{k,i} c_{k,i} > 0$, and we prove that $\overline{\dim}_{\rm B} E \leq d^*$.

For each given $t>d^*$, there exists $K>0$ such that for $k>K$, $d_k < t$. It follows that $\sum_{\bu\in \Sigma^k} c_\bu^t < 1$. Recall that  $\bu^-=u_1\ldots u_{k-1}$ for $\bu=u_1\cdots u_k\in\Sigma^k$.  For sufficiently small $r>0$, let
$$
\Sigma^*(r)=\{\bu \in \Sigma^*: c_\bu \leq r < c_{\bu^-}\}.
$$
Then $\{J_\bu: \bu\in \Sigma^*(r)\}$ is a $(r|J|)$-cover of $E$, and $c_* r < c_\bu$. Let $\sharp$ denote the cardinality of a set. Thus $N_{r|J|}(E)\leq \sharp \Sigma^*(r)$, and
$$
\sum_{\bu\in \Sigma^*(r)} c_\bu^t > (c_* r)^t \sharp \Sigma^*(r).
$$

Fix sufficiently small $r>0$. Let
$$
K_1=\max\{|\bu|: \bu\in \Sigma^*(r)\},\qquad K_2=\min\{|\bu|: \bu\in \Sigma^*(r)\},
$$
where $K_1 \geq K_2 > K$.

For each $\bu=u_1u_1\ldots u_{K_1}\in \Sigma^*(r)$, it is clear that $\bu^{-}j=u_1u_1\ldots u_{K_1-1}j\in \Sigma^*(r)$ for all $j\in\{1,\ldots,n_{K_1}\}$. If $\sum_j c_{K_1,j}^t \geq 1$, we have that
$$
\sum_j c_{\bu^* j}^t  \geq c_{\bu^*}^t \left(\sum_j c_{K_1,j}^t \right) \geq c_{\bu^*}^t .
$$
otherwise for $\sum_j c_{K_1,j}^t < 1$, we write that  $A_1=\{\bu\in \Sigma^*(r):|\bu|=K_1-1\}\cup \{\bu^-:\bu\in \Sigma^*(r), |\bu|=K_1\}$,  and obtain that
\begin{eqnarray*}
\sum_{\bu\in \Sigma^*(r) \atop |\bu|=K_1-1} c_{\bu}^t + \sum_{\bu\in \Sigma^*(r) \atop |\bu|=K_1} c_{\bu}^t &\geq& \sum_{\bu\in \Sigma^*(r) \atop |\bu|=K_1-1} c_{\bu}^t \left(\sum_j c_{K_1,j}^t\right)+\sum_{\bu\in \Sigma^*(r) \atop |\bu|=K_1} c_{\bu}^t \\
&\geq& \sum_{\bu\in A_1} c_{\bu}^t \left(\sum_j c_{K_1,j}^t t\right).
\end{eqnarray*}
By repeating this process, we have
$$
\sum_{\bu\in \Sigma^*(r)} c_\bu^t\leq \sum_{\bu\in \Sigma^k} c_\bu^t<1,
$$
for some $K_1 \leq k \leq K_2$. Hence $N_{r|J|}(E)\leq \sharp \Sigma^*(r) < (c_* r)^{-t}$, and it implies
$\overline{\dim}_{\rm B} E \leq d^*$.
\end{proof}

Next, we show that  $d_*$ and $d^*$ given by ~\eqref{defd} are the lower bounds of  Hausdorff dimension and upper  box-counting dimension of non-autonomous set $E$, respectively, if $E$  satisfies GSC, and for all $ k\geq 1$,
$$
 |S_{k,i}(x)-S_{k,i}(y)| \geq c_{k,i}|x-y|, \qquad (x,y)\in J
$$
where $c_{k,i}<1$,  for all $i=1,2,\ldots, n_k$.

\begin{proof}[Proof of Theorem~\ref{dimlb}]
First, we prove that $\dimh E \geq d_*$. For each $t<d_*$, there exists $K>0$ such that for $k>K$, $t<d_k$, which implies that $\prod_{l=1}^k (\sum_{j=1}^{n_l} c_{l,j}^t) > 1$.

Let $p_{k,j}=\frac{c_{k,j}^t}{\sum_{i=1}^{n_k} c_{k,i}^t}$,  and $\mu$ be a probability Borel measure on $\Sigma^\infty$ defined by
$$
\mu([\mathcal{C}_\bu])=p_{1,u_1}\ldots p_{k,u_k}.
$$
where $\bu=u_1 \ldots u_k$. Then $\widetilde{\mu}=\mu\circ \Pi^{-1}$ is a probability Borel measure on $E$.

For each  $x\in E$, there is a unique $\bu \in \Sigma^k$ such that $x\in J_\bu$ for each $k$. Fix $r>0$, we write
$$
\Sigma^*(r)=\{\bu\in \Sigma^*: C c_\bu |J| \leq r < C c_{\bu^-}|J|\} ,
$$
where $C$ is the constant in the definition of gap separation condition.
Given  $x\in E$, there exists $\bu\in \Sigma^*(r)$ such that $x\in J_\bu$ . Since $E $ satisfies gap separation condition, we have that $d(J_\bu, J_\bv) \geq C|J_{\bu^-}| \geq
C c_{\bu^-} |J|>r$, and it implies  that $E\cap B(x,r)\subset J_\bu$. Hence
\begin{eqnarray*}
\widetilde{\mu}(E\cap B(x,r)) &\leq& \widetilde{\mu}(J_\bu)\\
&=& \mu([\mathcal{C}_\bu]) \\
&=& \frac{c_{1,u_1}^t \ldots c_{k,u_k}^t}{\prod_{l=1}^k (\sum_{j=1}^{n_l} c_{l,j}^t)} \\
&\leq& c_{1,u_1}^t \ldots c_{k,u_k}^t \\
&\leq& C^{-t} |J|^{-t} r^t.
\end{eqnarray*}
For each $U\subset  \R^{d} $ such that $U\cap E\neq \emptyset$ and $|U|\leq Cc_*|J|$, we have that $U\subset B(x,|U|)$ for some $x\in E$, and  it implies that $\widetilde{\mu}(U) \leq C^{-t} |J|^{-t} |U|^t$. By the Mass distribution principle, see~\cite[Theorem 4.2]{Bk_KJF2}, we have that $\dimh E \geq t$. Since $t$ is arbitrarily chosen, it follows that $\dimh E \geq d_*$.

Next, we prove that  $\overline{\dim}_{\rm B} E \geq d^*$.
Fix $s<t<d^*$. There exists a sequence of $\{k_i\}$ such that $d_{k_i}>t$, and it implies that
$$
\sum_{\bu\in \Sigma^{k_i}} c_\bu^{t} > 1.
$$
For each large $i$, it is clear that
$$
\Sigma^{k_i}=\bigcup_{l=0}^\infty \{\bu\in \Sigma^{k_i}: 2^{-l-1}<c_\bu\leq 2^{-l}\}.
$$
Let $n_l$ be cardinality of the set $\{\bu\in \Sigma^{k_i}: 2^{-l-1}<c_\bu\leq 2^{-l}\}$.  Then there exists  $l_i$ such that
$$
n_{l_i} > 2^{l_i s}(1-2^{s-t}),
$$
and for all $\bu\neq\bv\in\Sigma^{k_i}$, by the gap separation condition, the gap between $J_\bu$ and $J_\bv$ is at least
$$
\inf\{|x -y |: x\in J_{\bu }, y\in J_{\bv}) \geq C |J_{\bu\wedge\bv}|\geq C |J_{\bu^-}|>2^{-l_i-1} C |J|.
$$
This implies that for every $x\in E$, the ball $B(x,2^{-l_i-1} C |J|)$ intersects only one basic set $J_\bu$ at $k_i$th level with $c_\bu > 2^{-l_i-1}$. Hence $N_{2^{-l_i-1} C |J|}(E) > n_{l_i}$, and immediately, we have that
\begin{eqnarray*}
  \limsup_{\delta\to 0} N_\delta(E) \delta^s &\geq & \limsup_{i\to \infty} N_{2^{-l_i-1} C |J|}(E)\cdot (2^{-l_i-1} C |J|)^s \\
  &\geq& \limsup_{i\to\infty} n_{l_i} 2^{-l_is}2^{-s} C^s |J|^s\\
 &\geq& 2^{-s} C^s |J|^s (1-2^{s-t})
\end{eqnarray*}
It follows that $\overline{\dim}_{\rm B} \geq s$ for all $s<d^*$, and the conclusion holds.
\end{proof}

\section{Box-counting  dimensions of non-autonomous affine sets}\label{sec 3}\label{secPB}
Recall that for each $s>0$ and $0<\epsilon <1$, let $m$ be the integer that $m-1<s\leq m$ and
$$
\Sigma^*(s, \epsilon)=\{\mathbf{u}=u_1\ldots u_k \in \Sigma^*: \alpha_m(T_\mathbf{u})  \leq \epsilon < \alpha_m(T_{\bu^-})  \}.
$$
For all real $s'$ such that  $m-1<s'\leq m$, we always have that
\begin{equation}\label{sigmaeq}
\Sigma^*(s', \epsilon)=\Sigma^*(s, \epsilon)=\Sigma^*(m, \epsilon).
\end{equation}
Since for each $\bu\in \Sigma^*(m, \epsilon)$, $\phi^s(T_\bu)$ is continuous and strictly decreasing, it is clear that
$$
\sum_{\bu\in\Sigma^*(s, \epsilon)}\phi^s(T_\bu)<\infty
$$
is continuous and strictly decreasing on $(m-1,m]$.

Let $N_\epsilon(A)$ be the smallest number of sets with diameters at most $\epsilon$ covering the set  $A\subset \mathbb{R}^d$. First we give the proof for that $s^*$ is the upper bound for the upper box dimension of $E$.

\begin{proof}[Proof of Theorem~\ref{dimBU}]
 Let $B$ be a ball such that $J\subset B$. Given $\delta > 0$, there exists an integer $k$ such that $|\Psi_\mathbf{u}(B)| < \delta$ for all $|\mathbf{u}| \geq k$. Let $A$ be a covering set of $\Sigma^\infty$ such that $|\mathbf{u}| \geq k$ for each $\mathbf{u}\in A$. Then $E \subset \bigcup_{\mathbf{u}\in A} \Psi_\mathbf{u}(B)$. Each ellipsoid $\Psi_\mathbf{u}(B)$ is contained in a rectangular parallelepiped of side lengths $2|B|\alpha_1, \ldots, 2|B|\alpha_d$ where the $\alpha_i$ are the singular values of $T_\mathbf{u}$.

 Fix $s>0$. Let $m$ be the least integer greater than or equal to $s$. We divide such a parallelepiped into at most
$$
(4|B|)^d\alpha_1 \alpha_2 \ldots \alpha_{m-1} \alpha_m^{1-m}=(4|B|)^d \phi^s(T_\mathbf{u}) \alpha_m^{-s}
$$
cubes of side $\alpha_m \leq \epsilon$.  Hence, the total number of cubes with side $\epsilon$ covering  $\bigcup_{\mathbf{u}\in \Sigma^*(s,\epsilon)} \Psi_\mathbf{u}(B)$ is bounded by
$$
(4|B|)^d \sum_{\mathbf{u}\in \Sigma^*(s,\epsilon)} \phi^s(T_\mathbf{u}) \alpha_m^{-s} \leq  (4|B|)^d  \alpha_-^{-s} \epsilon^{-s} \sum_{\mathbf{u}\in \Sigma^*(s,\epsilon)} \phi^s(T_\mathbf{u}).
$$
Since  $E\subset \bigcup_{\mathbf{u}\in \Sigma^*(s,\epsilon)} \Psi_\mathbf{u}(B)$,  it is clear that
\begin{equation}\label{Nlphi}
N_\epsilon(E) \leq (4|B|)^d  \alpha_-^{-s} \epsilon^{-s} \sum_{\mathbf{u}\in \Sigma^*(s,\epsilon)} \phi^s(T_\mathbf{u}).
\end{equation}

Since we only know that $\sum_{\bu\in\Sigma^*(s, \epsilon)}\phi^s(T_\bu)$
is continuous and strictly decreasing on $(m-1,m]$, we consider the following two cases:
$$
(1) \quad \limsup_{\epsilon \to \infty} \sum_{\mathbf{u}\in \Sigma^*(s^*,\epsilon)} \phi^{s^*}(T_\mathbf{u})<\infty, \qquad (2) \limsup_{\epsilon \to \infty} \sum_{\mathbf{u}\in \Sigma^*(s^*,\epsilon)} \phi^{s^*}(T_\mathbf{u})=\infty.
$$

\noindent Case (1): Since  $\limsup_{\epsilon\to 0}\sum_{\bu\in\Sigma^*(s^*, \epsilon)}\phi^{s^*}(T_\bu)<\infty$,  there exists a constant $\epsilon_0$ such that for $\epsilon<\epsilon_0$,
\begin{equation}
\sum_{\mathbf{u}\in \Sigma^*(s^*,\epsilon)} \phi^{s^*}(T_\mathbf{u}) < c,
\end{equation}
where $c$ is a constant. By \eqref{Nlphi}, it follows that
$$
N_\epsilon(E) \leq c_1 \epsilon^{-s^*},
$$
where $c_1=c (4|B|)^d \alpha_-^{-s^*}$.Hence
$$
-\frac{\log N_\epsilon(E)}{\log \epsilon} \leq \frac{ c_1 + s^*\log \epsilon }{\log \epsilon},
$$
and it implies  that
$$
\overline{\dim}_{\rm B} E = \limsup_{\epsilon\to 0} \frac{-\log N_\epsilon(E)}{\log \epsilon} \leq s^*.
$$
Therefore,  the conclusion holds.

\noindent Case (2): Suppose that $\limsup_{\epsilon\to 0} \sum_{\bu \in \Sigma^* (s^*, \epsilon)} \phi^{s^*}(T_\bu)=\infty$. Let $m$ be the least integer greater than $s^*$, and let $s$ be non-integral with $m-1 \leq s^* < s < m$. By the definition of $s^*$, there exists $s_0$ such that $s^*<s_0<s$ and
$$
\limsup_{\epsilon\to 0} \sum_{\bu \in \Sigma^* (s_0, \epsilon)} \phi^{s_0}(T_\bu)<\infty.
$$
Since $m-1<s_0<s<m$, we have that $\Sigma^* (s_0, \epsilon)=\Sigma^* (s, \epsilon)$. Hence for each $\epsilon>0$,
$$
\sum_{\bu \in \Sigma^* (s, \epsilon)} \phi^s(T_\bu)<\sum_{\bu \in \Sigma^* (s_0, \epsilon)} \phi^{s_0}(T_\bu).
$$
It follows that
$$
\limsup_{\epsilon\to 0} \sum_{\bu \in \Sigma^* (s, \epsilon)} \phi^s(T_\bu)\leq \limsup_{\epsilon\to 0} \sum_{\bu \in \Sigma^* (s_0, \epsilon)} \phi^{s_0}(T_\bu)<\infty.
$$
By the similar argument as above, we have that
$$
-\frac{\log N_\epsilon(E)}{\log \epsilon} \leq \frac{ c + s_0\log \epsilon }{\log \epsilon},
$$
and it implies that
$\overline{\dim}_{\rm B} E  \leq s_0<s$. Since $s$ is chosen arbitrarily,  we have that
$$
\overline{\dim}_{\rm B} E  \leq s^*,
$$
and the conclusion holds.

\end{proof}

To prove that $s^*$ is the lower bound of upper box dimension, we need the following technical lemma which gives a sufficient condition for the upper box dimension of $E$.
\begin{lem}\label{dimBL}
Let $E$ be the non-autonomous affine set given by ~\eqref{attractor}. Suppose that there exists an increasing sequence $\{s_n\}_{n=1}^\infty$ convergent to $s^*$  and a sequence $\{c_n>0\}_{n=1}^\infty$  such that for each integer $n>0$,
$$
N_\epsilon(E) \geq c_n \epsilon^{-s_n} \sum_{\mathbf{u}\in \Sigma^*(s_n,\epsilon)} \phi^{s_n}(T_\mathbf{u}),
$$
for all $\epsilon>0$.  Then $\overline{\dim}_{\rm B} E = s^*$.
\end{lem}

\begin{proof}
Since $\{s_n\}_{n=1}^\infty$ is increasing and  convergent to $s^*$, we assume that  $s_n $ is non-integral such that  $m-1<s_n<s^*\leq m$ for some integer $m$. By the definition of $s^*$ and $\Sigma^*(s_n, \epsilon)$, we have that
$$
\limsup_{\epsilon\to 0}\sum_{\bu\in\Sigma^*(s_n, \epsilon)}\phi^{s_n}(T_\bu)=\infty.
 $$
Hence there exists a sequence $\{\epsilon_{n_k}\}$ such that $\lim_{k\to \infty} \epsilon_{n_k} = 0$, and
\begin{equation}
\sum_{\mathbf{u}\in \Sigma^*(s_n,\epsilon_{n_k})} \phi^{s_n}(T_\mathbf{u}) > 1.
\end{equation}

Since
$$
N_{\epsilon_{n_k}}(E) \geq c_n\sum_{\mathbf{u}\in \Sigma^*(s_n\epsilon_{n_k})} \phi^{s_n}(T_\mathbf{u}) \epsilon_n^{-s_n} > c_n\epsilon_n^{-s_n},
$$
we have that
$$
\limsup_{k\to\infty}-\frac{\log N_{\epsilon_{n_k}}(E)}{\log \epsilon_{n_k}} \geq s_n.
$$
It follows that
$$
\overline{\dim}_{\rm B} E = \limsup_{\epsilon\to 0} \frac{-\log N_\epsilon(E)}{\log \epsilon} \geq s_n,
$$
for all $n>0$. Since $\lim_{n\to\infty} s_n= s$, letting $n$ tends to infinite, the conclusion holds.
\end{proof}

\begin{proof}[Proof of Corollary ~\ref{cor1}]
Since $T_{k,i}$ are scalar matrices, for $\bu=u_1\ldots u_k\in\Sigma^k$, the matrix $T_\bu$ is still a scalar matrix given by
$$
T_\mathbf{u}=\diag\{\alpha_{1,u_1} \alpha_{2,u_2} \ldots \alpha_{k,u_k},\ldots,\alpha_{1,u_1} \alpha_{2,u_2} \ldots \alpha_{k,u_k}\}.
$$
We write
$$
\alpha_\bu=\alpha_{1,u_1} \alpha_{2,u_2} \ldots \alpha_{k,u_k},
$$
and it is clear that $\phi^s(T_\mathbf{u})=\alpha_\bu^s$ and
$$
\Sigma^*(s,\epsilon)=\{\bu\in\Sigma^*:  \alpha_\bu \leq\epsilon< \alpha_{\bu^-}  \}.
$$
Note that $\Sigma^*(s,\epsilon)$ is independent of $s$.

Fix a sufficiently small $\epsilon>0$. By~\eqref{alpha}, for each $\bu\in\Sigma^*(s, \epsilon)$,
$$
\alpha_\bu \geq \alpha_- \alpha_{\bu^-} > \alpha_- \epsilon.
$$
There exists a constant $c_1$ such that every ball with radius $\epsilon$ intersects no more than $c_1$ basic sets $J_\bu $ for $\bu\in\Sigma^*(s,\epsilon)$. Hence $\card\,\Sigma^*(s,\epsilon) \leq c_1 N_\epsilon(E)$. It follows that
\begin{eqnarray*}
N_\epsilon(E) &\geq& c_1^{-1} \card\,\Sigma^*(s,\epsilon)\\
&\geq& c_1^{-1} \sum_{\bu\in\Sigma^*(s,\epsilon)} (\alpha_-)^{-1} \alpha_\bu \epsilon^{-s}\\
&\geq& c \epsilon^{-s}\sum_{\mathbf{u}\in \Sigma^*(s,\epsilon)} \phi^s(T_\mathbf{u}),
\end{eqnarray*}
where $c$ is a constant independent of $\bu$ and $\epsilon$.
By Lemma ~\ref{dimBL}, we have that
$$
\overline{\dim}_{\rm B} E \geq s^*.
$$
Combining with Theorem~\ref{dimBU}, the conclusion holds.
\end{proof}

\begin{proof}[Proof of Theorem~\ref{projc}]
By Theorem~\ref{dimBU}, it is clear that $s^*\geq \overline{\dim}_{\rm B}(E)\geq d-1$. If $s^*=d-1$, it is clear that
$$
\overline{\dim}_{\rm B} E =d-1=s^*,
$$
and the conclusion holds.  Hence we only need to show $\overline{\dim}_{\rm B} E \geq s^*$ for $s^*>d-1$.

Let  $U$ be  the open set in the OPC condition. Given $\delta >0$, we write
$$
U_{-\delta}=\{x\in U: B(x,\delta)\in U\}.
$$
For each given $(d-1)$-dimensional subspace $\Theta$, by the continuity of Lebesgue measure, it is clear that   $\mathcal{L}^{d-1}(\proj_\Theta U_{-\delta})$ monotonically increases to $\mathcal{L}^{d-1}(\proj_\Theta U)$ as $\delta$ tends to $0$. Since $\mathcal{L}^{d-1}(\proj_\Theta U_{-\delta})$ and $\mathcal{L}^{d-1}(\proj_\Theta U)$ are continuous in $\Pi$ with respect to  the Grassmann topology on the set of $(d-1)$-dimensional subspaces, by Dini's theorem, we have that $\mathcal{L}^{d-1}(\proj_\Theta U_{-\delta})$  uniformly converges to  $\mathcal{L}^{d-1}(\proj_\Theta U)$ in $\Theta$. Since $\mathcal{L}^{d-1}(\proj_\Theta U)=\mathcal{L}^{d-1}(\proj_\Theta \overline{U})$,  we may choose $\delta > 0$ such that
$$
\mathcal{L}^{d-1}(\proj_\Theta U_{-\delta}) \geq \mathcal{L}^{d-1}(\proj_\Theta \overline{U})-\frac{1}{2}c.
$$
for all subspaces $\Theta$. Let $G_0=\proj_\Theta U_{-\delta}\cap \proj_\Theta (J\cap \Psi_\bu^{-1}(E))$, then
$$
\mathcal{L}^{d-1}(G_0) \geq \mathcal{L}^{d-1}(\proj_\Theta (J\cap \Psi_\bu^{-1}(E))) - \mathcal{L}^{d-1}(\proj_\Theta \overline{U})+ \mathcal{L}^{d-1}(\proj_\Theta U_{-\delta}) \geq \frac{1}{2}c.
$$
For each $\mathbf{u}\in \Sigma$, if $\Theta$ is the
$(d-1)$-dimensional subspace of $\R^d$ perpendicular to the shortest semi-axis of the ellipsoid $\Psi_ \mathbf{u}(B)$, where $B$ is the unit ball in $\R^d$.
Let $G=\Psi_\bu(G_0)$, then
\begin{equation}\label{ineqd-1}
\mathcal{L}^{d-1}(G) \geq \frac{1}{2}c\alpha_1(T_\mathbf{u})\ldots \alpha_{d-1}(T_\mathbf{u}).
\end{equation}
for  $y=\Psi_\bu(y_0)\in G$, then there exists $x_0\in U_{-\delta}$ such that $\proj_{\Psi_\bu^{-1}\Pi} x_0=y_0$. Let $x=\Psi_\bu(x_0)$, then $\proj_\Theta x=y$ and $B(x,\delta\alpha_d(T_\bu)) \subset \Psi_\bu(U)$.
Hence
\begin{equation}\label{ineqd=1}
\mathcal{L}^1\{x\in \Psi_u(U): \proj_\Theta x=y\} \geq \delta \alpha_d(T_\mathbf{u}),
\end{equation}
for all $y\in G$. For all $x\in \Psi_ \mathbf{u}(U)$ such that $\proj_\Theta x \in G$, it is clear that the distance from $x$ to the set $E$ is no more that$ |U| \alpha_d(T_\mathbf{u})$. Combining \eqref{ineqd-1} and \eqref{ineqd=1} together, we obtain that
$$
\mathcal{L}^d\{x\in \Psi_\mathbf{u}(U)  :|x-z|\leq |U| \alpha_d(T_\mathbf{u}) \text{ for some } z\in E\} \geq \frac{1}{2}c\delta \alpha_1(T_\mathbf{u})\ldots \alpha_d(T_\mathbf{u}).
$$

Arbitrarily choose $0<\epsilon < 1$. For each $d-1<s<s^*$, the set $\Sigma^*(s,\epsilon)$ is independent of $s$. Note that OPC condition implies that   $\bigcup_{\mathbf{u}\in \Sigma^*(s,\epsilon)}\Psi_\bu(U)$ is a union of disjoint open sets.
For each $s^*< s \leq d$,
\begin{eqnarray*}
&& \mathcal{L}^d \{x\in \R^d :|x-z|\leq |U|\epsilon \text{ for some } z\in E\} \\
&\geq&  \sum_{\mathbf{u}\in \Sigma^*(s, \epsilon)} \mathcal{L}^d \{x\in \Psi_\bu(U) :|x-z|\leq |U|\alpha_d(T_\mathbf{u}) \text{ for some } z\in E\}\\
&\geq& \frac{1}{2}c\delta\alpha_-^{d-s}\epsilon^{d-s}\sum_{\mathbf{u}\in \Sigma^*(s, \epsilon)} \phi^s(T_\mathbf{u}).
\end{eqnarray*}

Note that the set $\{x\in \R^d :|x-z|\leq \epsilon \text{ for some } z\in E\}$ may be covered by $N(\epsilon)$ balls of radius $2\epsilon$ if $E$ is covered by $N(\epsilon)$ balls of radius $\epsilon$.  Therefore,
$$
\mathcal{L}^d \{x\in \R^d :|x-z|\leq \epsilon \text{ for some } z\in E\} \leq N(\epsilon)c_1 (2\epsilon)^d,
$$
where $c_1$ is the volume of the d-dimensional unit ball. It follows that
$$
N_{|U|\epsilon}(E) \geq 2^{-d-1}|U|^{-d}c_1^{-1}c\delta \alpha_-^{d-s}\epsilon^{-s} \sum_{\mathbf{u}\in \Sigma^*(\epsilon)} \phi^s(T_\mathbf{u}).
$$
By Lemma ~\ref{dimBL}, we obtain that $\overline{\dim}_{\rm B} E =  s^*$, and the conclusion holds.
\end{proof}

Next,  we prove that $s_A$ is a lower bound for the lower box-dimension of the non-autonomous affine set $E$.
\begin{proof}[Proof of Theorem~\ref{dimLBL}]
Since
$$
\mathcal{L}^{d-1}\{E\} \geq \mathcal{L}^{d-1}\{\proj_\Theta (J\cap \Psi_\bu^{-1}(E))\}\geq c
$$
for all $(d-1)$-dimensional subspaces $\Theta$, the Hausdorff dimension of $E$ is at least $d-1$. By Theorem ~\ref{dimHU}, $s_A \geq \dimh E \geq d-1$. If $s_A=d-1$, then
$$
\underline{\dim}_{\rm B} E \geq \dimh E \geq d-1=s_A.
$$
Hence it is sufficient to show $\underline{\dim}_{\rm B} E \geq s_A$ for s$_A>d-1$.

For every $d-1<s<s_A$, we have that  $\mathcal{M}^s(E)=\infty$. Hence there exists a constant $c_1>0$ such that for all sufficiently small $\epsilon>0$,
$$
\sum_{\mathbf{u}\in \Sigma^*(s, \epsilon)}\phi^s(T_\mathbf{u}) \geq c_1.$$
By the same argument as in Theorem~\ref{projc}, we have that
$$
\mathcal{L}^d \{x\in \R^d :|x-w|\leq |U|\epsilon \text{ for some } w\in E\}\geq \frac{1}{2}c\delta\alpha_-^{d-s}\epsilon^{d-s}\sum_{\mathbf{u}\in \Sigma^*(s, \epsilon)} \phi^s(T_\mathbf{u}).
$$
This implies that
$$
\mathcal{L}^d \{x\in \R^d :|x-w|\leq |U|\epsilon\text{ for some } w\in E\} \geq c_2\epsilon^{d-s},
$$
where $c_2=\frac{1}{2}c\delta\alpha_-^{d-s}$. Using the Minkowski definition of box-counting
dimension, see \cite{Bk_KJF2}, it follows that
$$
\underline{\dim}_{\rm B} E = d-\limsup_{r\to 0}\frac{\log \mathcal{L}^d \{x\in \R^d :|x-w|\leq r \text{ for some } w\in E\}}{\log r} \geq s,
$$
for all $s\leq s_A$. Hence $\underline{\dim}_{\rm B} E\geq  s_A$, and the conclusion holds.
\end{proof}

\begin{proof}[Proof of Corollary ~\ref{cor2}]
Since $E$ has a connected component that is not contained in any straight line, we choose  three non-collinear points $x_1,x_2,x_3$  in the same connected component of $E$. We write $c=\min\{\|x_i-x_j\| : i\neq j \}$.   Since $J\cap \Psi_\bu^{-1}(E)=E$, it implies that
$$
\mathcal{L}^1\{\proj_\Theta (J\cap \Psi_\bu^{-1}(E))\}=\mathcal{L}^1\{\proj_\Theta E\},
$$
and it is clear that
$$
\mathcal{L}^1\{\proj_\Theta (J\cap \Psi_\bu^{-1}(E))\}\geq c.
$$
By Theorem~\ref{projc} and Theorem~\ref{dimLBL}, the conclusion holds.

\end{proof}

\section{Hausdorff dimension of non-autonomous affine sets}\label{secH}

First, we give the proof that $s_A$ is an upper bound for the Hausdorff dimension of non-autonomous affine sets.
\begin{proof}[Proof of Proposition ~\ref{dimHU}]
Let $B$ be a sufficiently large ball such that $J\subset B$. Given $\delta > 0$, there exists an integer $k$ such that $|\Psi_\mathbf{u}(B)| < \delta$ for all $\bu$ such that  $|\mathbf{u}| \geq k$. Let $A$ be a covering set of $\Sigma^\infty$ such that $|\mathbf{u}| \geq k$ for each $\mathbf{u}\in A$. Then $E \subset \bigcup_{\mathbf{u}\in A} \Psi_\mathbf{u}(B)$. Each ellipsoid $\Psi_\mathbf{u}(B)$ is contained in a rectangular parallelepiped of side lengths $2|B|\alpha_1, \ldots, 2|B|\alpha_d$ where the $\alpha_i$ are the singular values of $T_\mathbf{u}$.

For each $s>s_A$, let $m$ be the least integer greater than or equal to $s$. We can divide such a parallelepiped into at most
$$
\Big(4|B|\frac{\alpha_1}{\alpha_m}\Big) \Big(4|B|\frac{\alpha_2}{\alpha_m}\Big)\ldots \Big(4|B|\frac{\alpha_{m-1}}{\alpha_m}\Big)\Big(4|B|\Big)^{d-m+1}
$$
cubes of side $\alpha_m$.

Taking such a cover of each ellipsoid $\Psi_\mathbf{u}(B)$ with $\mathbf{u}\in A$, it follows that
\begin{eqnarray*}
\mathcal{H}_{\sqrt{d}\delta}^s(E)&\leq& \sum_{\mathbf{u}\in A} (4|B|)^d\alpha_1 \alpha_2 \ldots \alpha_{m-1} \alpha_m^{1-m}(\sqrt{d}\alpha_m)^s\\
&\leq& (4|B|\sqrt{d})^d \sum_{\mathbf{u}\in A} \phi^s(T_\mathbf{u}).
\end{eqnarray*}
Since this holds for every covering set $A$ with $|\mathbf{u}| \geq k$ for $\mathbf{u}\in A$, we have that
$$
\mathcal{H}_{\sqrt{d}\delta}^s(E) \leq (4|B|\sqrt{d})^d \mathcal{M}_{(k)}^s(\Sigma^\infty).
$$
Letting $\delta$ tend to $0$, it follows that
$$
\mathcal{H}^s(E) \leq (4|B|\sqrt{d})^d \mathcal{M}^s(\Sigma^\infty) < \infty.
$$
Hence $\dimh E \leq s$ for all $s>s_A$, which implies that $\dimh E \leq s_A$.
\end{proof}

The following cited theorem gives a necessary and sufficient condition for absolute continuity of measures, which is very useful in our proofs, and we refer readers to \cite[Theorem 2.12]{Book_mattila} for details.
\begin{thm}\label{abcon}
Let $\mu$ and $\lambda$ be Radon measures on $\R^d$. Then $\mu \ll \lambda$ if and only if
$$
\liminf_{r\to 0} \frac{\mu(B(x,r))}{\lambda(B(x,r))}  <\infty,
$$
for $\mu$-almost all $x\in\R^d$.
\end{thm}

To show that $s_A$ is also a lower bound for the Hausdorff dimension of non-autonomous affine sets, we need the following lemmas for the connection between integral estimates and singular value functions.

The first lemma shows that singular value functions provide useful estimates for potential integrals, which was proved by Falconer in~\cite[Lemma 2.1]{Falco88}. Let $\overline{B(0,\rho)}$ be the closed ball in $\R^d$ with centre at the origin and radius $\rho$.

\begin{lem}\label{lemtphi}
Let $s$ satisfy $0 < s\leq  n$ with $s$ non-integral.
Then there exists $c>0$ such that, for all non-singular
linear transformations $T\in {\mathcal{ L}}(\R^d,\R^d)$,
$$
\int_{\overline{B(0,\rho)}} \frac{dx}{|Tx|^{s}}
\leq \frac{c}{\phi^{s}(T)}.
$$
\end{lem}

The next lemma provides  a technical result on net measures.
\begin{lem}\label{measure_ineq}
Let $\mathcal{M}^s$ be the Borel measure defined by ~\eqref{measure}. If
$\mathcal{M}^s(\Sigma^\infty)>0$ then there exists a
Borel measure $\mu$ on $\Sigma^\infty$ such that
$0<\mu(\Sigma^\infty)<\infty$  and  a constant  $c>0$ such that
\begin{equation}\label{frost}
    \mu(\mathcal{C}_\bu)\leq c \phi^s(T_\mathbf{u}) ,
\end{equation}
for all  $ \bu \in \Sigma^*$.
\end{lem}

\begin{proof}
Since $\mathcal{M}^s(\Sigma^\infty)>0$, by a similar argument to Theorem 5.4 in \cite{Book_KJF1}, there exists a compact $G \subset \Sigma^\infty$ such that $\mathcal{M}^s(G)>0$ and
$$
\mathcal{M}^s(G\cap \mathcal{C}_\bu) \leq c\phi^s(T_\mathbf{u}),
$$
for all $\bu \in \Sigma^*$, where $c>0$ is a constant independent of $\bu$.  (Alternatively, This may be regarded as  a special case of Theorem 54 in \cite{Bk_Rogers}.)

We write $\mu$ for the measure given  by
$$
\mu(A) = \mathcal{M}^s(G\cap A),
$$
for all Borel $A \subset \Sigma^{\infty}$, and the measure $\mu$  has the desired properties.
\end{proof}

Recall that  $\Gamma=\{a_1,\ldots, a_\tau\}$ is a finite collection of translations, where $a_1,\ldots, a_\tau$ are regarded as variables in $\R^d$. For each $\mathbf{u}=u_1\ldots u_k \in \Sigma^*$,  $J_{\mathbf{u}}=\Psi_{\mathbf{u}}(J)=\Psi_{u_1}\circ  \ldots \circ \Psi_{u_k} (J)$. Suppose that the translation of $\Psi_{u_j}$ is an element of $\Gamma$, that is,
 $$
\Psi_{u_j}(x)=T_{j,u_j}x+ \w_{u_1\ldots u_j}, \qquad \w_{u_1\ldots u_j}\in \Gamma,
  $$
 for $j=1,2,\ldots, k$. We write $\ba=(a_1,\ldots, a_\tau)$ as a variable in $\R^{\tau d}$. To emphasize the dependence on these special translations in $\Gamma$,  we denote the non-autonomous affine set by $E^\ba$ and denote the projection $\Pi$  by $\Pi^\ba$, that is
\begin{eqnarray*}
\Pi^\ba (\bu) & =&   \w_{u_1} +T_{1,u_1}\w_{u_1u_2}+ \cdots + T_{\bu|k}\w_{\bu|k+1}   + \cdots,   \qquad \w_{\bu|k}\in \Gamma.
\end{eqnarray*}
It is clear that $E^\ba=\Pi^\ba(\Sigma^\infty)$. We write $\overline{\mathbf{B}(0,\rho)}$ for the closed ball in $\mathbf{R}^{\tau d}$ with centre the origin and radius $\rho$.

\begin{lem}\label{lemHD}
Let $s$ satisfy $0 < s < s_A$ with $s$ non-integral.
Suppose that there exists a constant $c>0$ such that, for all $\bu,\bv\in \Sigma^\infty$, $\bu \neq \bv$,
\begin{equation}\label{ineqPAP}
\int_{\overline{\mathbf{B}(0,\rho)}} \frac{d \ba}{|\Pi^\ba(\bu)-\Pi^\ba(\bv)|^{s}}
\leq \frac{c}{\phi^{s}(T_{\bu \wedge \bv})}.
\end{equation}
Then
$$
\dimh E^\ba \geq s_A,
$$
for  $\mathcal{L}^{\tau d}$-almost all $\ba$.
\end{lem}

\begin{proof}
For each $s<s_A$, choose $t$ such that $s<t<s_A$, then $\mathcal{M}^t(\Sigma^\infty)>0$. By Lemma~\ref{measure_ineq}, there exists a
Borel measure $\mu$ on $\Sigma^\infty$ such that
$0<\mu(\Sigma^\infty)<\infty$ and a constant $c_1>0$ such that
\begin{equation}\label{ineqmph}
\mu(\mathcal{C}_\bu) \leq c_1 \phi^t(T_\bu),
\end{equation}
for all $\bu \in \Sigma^*$.
By Tonelli Theorem and \eqref{ineqPAP}, we have that
\begin{eqnarray*}
\int_{\ba\in\overline{\mathbf{B}(0,\rho)}}\int_{\Sigma^\infty}\int_{\Sigma^\infty}
\frac{d\mu(\bu)d\mu(\bv)d\ba}{|\Pi^\ba(\bu)-\Pi^\ba(\bv)|^s}
&=& \int_{\Sigma^\infty}\int_{\Sigma^\infty}\int_{\ba\in
\overline{\mathbf{B}(0,\rho)}}
\frac{d\ba d\mu(\bu)d\mu(\bv)}{|\Pi^\ba(\bu)-\Pi^\ba(\bv)|^s}\\
&\leq & c\int_{\Sigma^\infty}\int_{\Sigma^\infty}
\frac{d\mu(\bu)d\mu(\bv)}{\phi^s(T_{\bu\wedge\bv})}
\label{intest}\\
&\leq& c\sum_{\mathbf{p}\in \Sigma^*} \sum_{\bu'\wedge \bv'= \emptyset}
\frac{\mu(\mathcal{C}_{\mathbf{p}\bu'})
\mu(\mathcal{C}_{\mathbf{p}\bv'})}{\phi^s(T_{\mathbf{p}})}\nonumber\\
&\leq& c\sum_{\mathbf{p}\in \Sigma^*}
\frac{\mu(\mathcal{C}_{\mathbf{p}})^2}{\phi^s(T_{\mathbf{p}})}.\nonumber
\end{eqnarray*}

Combining \eqref{ineqmph},\eqref{alhpha+-} and \eqref{aspalpha}, it follows that
\begin{eqnarray*}
\sum_{\mathbf{p}\in \Sigma^*}
\frac{\mu(\mathcal{C}_{\mathbf{p}})^2}{\phi^s(T_{\mathbf{p}})}&\leq& c_1 \sum_{k=0}^{\infty}\sum_{\mathbf{p}\in \Sigma^k}
\frac{\phi^t(T_{\mathbf{p}})\mu(
\mathcal{C}_{\mathbf{p}})}{\phi^s(T_{\mathbf{p}})}\nonumber\\
&\leq& c_1 \sum_{k=0}^{\infty}\sum_{\mathbf{p}\in \Sigma^k}
\alpha_{+}^{k(t-s)}\mu(\mathcal{C}_{\mathbf{p}})\nonumber\\
&\leq& c_1 \mu(\Sigma^\infty) \sum_{k=0}^{\infty}\alpha_{+}^{k(t-s)}\nonumber \\
&< & \infty.
\end{eqnarray*}
Therefore the following estimate
$$
\int_{\Sigma^\infty}\int_{\Sigma^\infty}
\frac{d\mu(\bu)d\mu(\bv)}{|\Pi^\ba(\bu)-\Pi^\ba(\bv)|^s}<\infty
$$
holds for $\mathcal{L}^{\tau d}$-almost all $\ba\in
\R^{\tau d}$. By Potential theoretic method, see \cite[Theorem 4.13]{Bk_KJF2}, we obtain that $\dimh E^\ba\geq s$ for $\mathcal{L}^{\tau d}$-almost all $\ba\in
\R^{\tau d}$. Since $s<s_A$ is arbitrarily chosen, the conclusion holds.
\end{proof}

Finally, we prove that $s_A$ gives the Hausdorff dimension of non-autonomous affine sets $E^\ba$ almost surely.
\begin{proof}[Proof of Theorem~\ref{dimHL}]
Recall that
$$
\Pi^\ba (\bu)  =   \w_{u_1} +T_{1,u_1}\w_{u_1u_2}+ \cdots +  T_{\bu|k}\w_{\bu|k+1}   + \cdots,   \qquad \w_{\bu|k}\in \Gamma,
$$
where $\Gamma=\{a_1,\ldots, a_\tau\}$. For $\bu, \bv\in \Sigma^\infty$, we assume that $|\bu \wedge \bv|=n$. Without loss of generality, suppose that $\w_{\bu|{n+1}}=a_1$, $w_{\bv|{n+1}}=a_2$. Then
\begin{eqnarray*}
\Pi^\ba(\bu)-\Pi^\ba(\bv)&=& T_{\bu\wedge\bv}\Big(\w_{\bu|n+1}-\w_{\bv|n+1}+ (T_{n+1,u_{n+1}}\w_{\bu|n+2} \\
&& +T_{n+1,u_{n+1}}T_{n+2,u_{n+2}} \w_{\bu|{n+3}}+\ldots)
-(T_{n+1,v_{n+1}}\w_{\bv|n+2} \\
&& +T_{n+1,v_{n+1}}T_{n+2,v_{n+2}} \w_{\bu|{n+3}} +\ldots)\Big) \\
&=& T_{\bu\wedge\bv}(a_1-a_2+ H(\ba)),
\end{eqnarray*}
where $H$ is a linear map from $\R^{\tau d}$ to $\R^d$. We may write $H(\ba)=\sum_{a_j\in\Gamma \atop 1\leq j\leq \tau} H_j(a_j)$.

We write
\begin{eqnarray*}
\eta &=&\sup\{\|T_{k,j}\| : T_{k,j}\in \Xi_k, 0<j\leq n_k,  k>0 \},
\end{eqnarray*}
and we have that  $\eta < \frac{1}{2}.$  Let
$$
m=\inf\{k\geq 2:\w_{\bu|{n+k}}=\w_{\bv|{n+k}}\}.
$$
If $m=\infty$, it is straightforward that
$$
\|H_1\|\leq \sum_{k=2}^\infty \eta^{k-1} < 1.
$$
Otherwise for $m<\infty$, it is clear that $\w_{\bu|{n+m}}$ and $\w_{\bv|{n+m}}$ can not  equal to $a_1$ and $a_2$ simultaneously. We assume that $\w_{\bu|{n+m}}=\w_{\bv|{n+m}}\neq a_1$. Since $\eta <\frac{1}{2}$,  it follows that
\begin{eqnarray*}
\|H_1\|&\leq& \sum_{k=2}^{m-1} \eta^{k-1}+ \sum_{k=m+1}^{\infty} 2\eta^{k-1}<1 
\end{eqnarray*}
This implies that  the linear transformation $I+H_1$ is invertible. We define  the  invertible linear transformation by
\begin{equation}\label{inlitr}
y=a_1-a_2+ H(\ba), \ a_2=a_2,\ldots,\ a_\tau=a_\tau.
\end{equation}

For $s_A\leq d$, arbitrarily choosing a non-integral $0<s<s_A $, by Lemma~\ref{lemtphi} and \eqref{inlitr}, it follows that
\begin{eqnarray*}
\int_{\overline{\mathbf{B}(0,\rho)}}\frac{d\ba}{|\Pi^\ba(\bu)-\Pi^\ba(\bv)|^s} &=& \int_{\overline{\mathbf{B}(0,\rho)}}\frac{d\ba}{|T_{\bu\wedge\bv} (a_1-a_2+ H(\ba))|^s}\\
&\leq& c_1 \int \ldots \int_{y\in B_{(2+k)\rho} \atop a_j\in B_\rho} \frac{dy da_2 \ldots}{|T_{\bu\wedge\bv}(y)|^s}\\
&\leq& \frac{c}{\phi^{s}(T_{\bu \wedge \bv})}.
\end{eqnarray*}
Immediately, Lemma ~\ref{lemHD} implies that
$$
\dimh E^\ba \geq s_A
$$
for $\mathcal{L}^{\tau d}$-almost all $\ba$. Combining  Theorem ~\ref{dimHU}, the equality
$$
\dimh E^\ba = s_A
$$
holds for $\mathcal{L}^{\tau d}$-almost all $\ba$.

Next, we prove that $\mathcal{L}^d(E^\ba)>0$ for $s_A>d$.
Arbitrarily choose $d<s<s_A$. Then by the definition of $s_A$, we have that $\mathcal{M}^s(\Sigma^\infty)=\infty$. By Lemma ~\ref{measure_ineq}, there exists a
Borel measure $\mu$ on $\Sigma^\infty$ such that
$0<\mu(\Sigma^\infty)<\infty$ and a constant $c_1>0$ such that
\begin{equation}\label{ineqmuphi}
\mu(\mathcal{C}_\bu) \leq c_1 \phi^s(T_\bu),
\end{equation}
for all $\bu \in \Sigma^*.$
Let  $\mu^\ba$ be the projection measure  of $\mu$, that is $\mu^\ba=\mu\circ (\Pi^\ba)^{-1}$.  It is clear that
$$\mu^\ba(E^\ba)=\mu(\Sigma^\infty)>0.$$

To prove that $\mathcal{L}^d(E^\ba)>0$, we just need to prove $\mu_\ba \ll \mathcal{L}^d$. By Theorem~\ref{abcon}, it is equivalent to show that for $\mu^\ba$-almost all $x$,
$$\liminf_{r\to 0} \frac{\mu^\ba(B(x,r))}{ r^d}<\infty.$$

For simplicity, we write $\alpha_1,\ldots,\alpha_d$ for the singular values of
$T_{\bu\wedge\bv}$. For all distinct $\bu,\bv\in \Sigma^\infty$ and $\rho>0$, by \eqref{inlitr}, we have that
\begin{eqnarray}
&&\mathcal{L}^{\tau d}\{\ba \in \R^{\tau d}: |\Pi^\ba(\bu)-\Pi^\ba(\bv)|\leq \rho\}  = \int \I_{\{\ba \in \R^{\tau d}:|T_{\bu\wedge\bv}(a_1-a_2+ H(\ba))|\leq \rho\}}d\ba \nonumber\\
&&\leq c \int \I_{\{(y,a_2,\ldots,a_\tau) \in \R^{\tau d}:|T_{\bu\wedge\bv}(y)|\leq \rho\}}d y da_2\ldots da_\tau \nonumber\\
&&\leq c \mathcal{L}^d \{T_{\bu\wedge\bv}^{-1}(B(0,\rho))\}    \nonumber\\
&&\leq c \mathcal{L}^d \Big\{(x_1,\ldots,x_d)\in\R^d: |x_1|\leq \frac{\rho}{\alpha_1},\ldots,|x_d|\leq \frac{\rho}{\alpha_d}\Big\}      \nonumber\\
&&= c\cdot \frac{\rho^d}{\phi^d(T_{\bu\wedge\bv})}, \nonumber
\end{eqnarray}
where $c$ is a constant. Applying Fatou's Lemma and Fubini's Theorem, this implies that
\begin{eqnarray*}
&&\iint \liminf_{r\to 0} \frac{\mu^\ba(B(x,r))}{ r^d}d\mu^\ba(x) d\ba\\
&\leq& \liminf_{r\to 0} \frac{1}{ r^d} \iiint \I_{\{(\bu,\bv):|\Pi^\ba(\bu)-\Pi^\ba(\bv)|\leq r\}} d\mu(\bu)d\mu(\bv)d\ba\\
&\leq& \liminf_{r\to 0} \frac{1}{ r^d} \iint \mathcal{L}^{\tau d}\{\ba \in \R^{\tau d}: |\Pi^\ba(\bu)-\Pi^\ba(\bv)|\leq \rho\} d\mu(\bu)d\mu(\bv)\\
&\leq& c \iint \frac{1}{\phi^d(T_{\bu\wedge\bv})}d\mu(\bu)d\mu(\bv).
\end{eqnarray*}

Furthermore, by~\eqref{ineqmuphi} and \eqref{alhpha+-}, we have that for $\bu\in\Sigma^k$,
\begin{equation}\label{ineq_muphi}
\frac{\mu(\mathcal{C}_\bu)}{\phi^d(T_\bu)}\leq \frac{c_1\cdot \phi^s(T_\bu)}{\phi^d(T_\bu)} \leq c_1 \phi^d(T_\bu)^\frac{s-d}{d} \leq c_1 \alpha_+^{(s-d)k}.
\end{equation}
By the similar argument as above, we have that
\begin{eqnarray*}
 \iint \frac{1}{\phi^d(T_{\bu\wedge\bv})}d\mu(\bu)d\mu(\bv)
&\leq& \sum_{\mathbf{p}\in \Sigma^*} \sum_{\bu'\wedge \bv'= \emptyset}
\frac{\mu(\mathcal{C}_{\mathbf{p}\bu'})
\mu(\mathcal{C}_{\mathbf{p}\bv'})}{\phi^d(T_{\mathbf{p}})}\\
&\leq& c_1\sum_{k=0}^\infty \sum_{\mathbf{p}\in \Sigma^k} \alpha_+^{(s-d)k} \mu(\mathcal{C}_{\mathbf{p}})\\
&<&\infty.
\end{eqnarray*}
Hence for $\mathcal{L}^{\tau d}$-almost all $\ba$,   the inequality
$$\liminf_{r\to 0} \frac{\mu^\ba(B(x,r))}{ r^d}<\infty$$
holds for $\mu^\ba$-almost all $x$. By Theorem~\ref{abcon}, this implies that  $\mu^\ba \ll \mathcal{L}^d$. Since $\mu^\ba(E^\ba)>0$, it implies that $\mathcal{L}^d(E^\ba)>0$ for $\mu^\ba$-almost all $x$, and  the conclusion holds.

\end{proof}

\section{Affine Moran set with random translations}\label{secASA}
Recall that $\mathcal{D}$ is a bounded region in $\R^d$. For each $\bu\in \Sigma^* $,  the translation $\omega_\bu\in  \mathcal{D} $ is an independent random vector identically distributed
according to the probability measure $P=P_\bu$ which is absolutely continuous with respect to $d$-dimensional Lebesgue measure. The product probability measure $\mathbf{P} $ on
the family  $\omega =\{\omega_\bu : \bu \in  \Sigma^* \}$ is given by
$$
\mathbf{P} = \prod_{\bu\in \Sigma^*} P_\bu.
$$
For each $\mathbf{u}=u_1\ldots u_k \in \Sigma^*$, let $J_{\mathbf{u}}=\Psi_{\mathbf{u}}(J)=\Psi_{u_1}\circ  \ldots \circ \Psi_{u_k} (J)$, where the translation of $\Psi_{u_j}$ is an element of $\w$, that is,
 $$
\Psi_{u_j}(x)=T_{j,u_j}x+ \w_{u_1\ldots u_j}, \qquad \w_{u_1\ldots u_j}\in \omega =\{\omega_\bu : \bu \in  \Sigma^* \},
  $$
 for $j=1,2,\ldots, k$.
Assume that the collection $\mathcal{J}=\{J_\bu: \bu\in \Sigma^*\}$ fulfils the non-autonomous structure.
Since
\begin{eqnarray*}
\Pi^\w(\bu) &=&\w_{u_1} +T_{1,u_1}\w_{u_1u_2} + \cdots+ T_{\bu|k}\w_{\bu|k+1}   + \cdots,
\end{eqnarray*}
where $\omega_{\bu| j}\in  \mathcal{D} $, j=1,2,3\ldots, are random vectors. The points $\Pi^\omega(\bu) \in \mathbb{R}^d$ are  random points whose aggregate form the random set $E^\omega$, that is,
\begin{equation*}
E^{ \w} = \bigcup_{\bu \in \Sigma^\infty} \Pi^\w(\bu) \subset
\R^d, \label{2.6}
\end{equation*}
and $E^\w$ is called  {\it non-autonomous affine set with random translations}.

Let $\mathbb{E}$ denote expectation. Given $ \Lambda \subset \Sigma^*$, we write $\mathcal{F} =\sigma\{\w_\bu : \bu\in \Lambda\}$ for the sigma-field generated by the random vectors $\w_\bu$ for $\bu\in \Lambda$ and write $\mathbb{E}(X | \mathcal{F})$ for the expectation of a random variable $X$ conditional on $\mathcal{F}$; intuitively this is the expectation of $X$ given all $ \{\w_\bu : \bu\in\Lambda\}$.

With respect to the probability setting, we have a similar conclusion to Lemma~\ref{lemHD} to  estimate the potential integral. Since the proof is almost identical, we omit it.
\begin{lem}\label{thmHD}
Let $s$ satisfy $0 < s < s_A$ with $s$ non-integral. Suppose that there exists $c>0$ such that
$$
\mathbb{E}\Big(|\Pi^\w(\bu)-\Pi^\w(\bv)|^{-s}\  \Big| \ \mathcal{F}\Big) \leq  \frac{c}{\phi^{s}(T_{\bu \wedge \bv})},
$$
for all $\bu,\bv\in \Sigma^\infty$, $\bu \neq \bv$, where $\mathcal{F} =\sigma\{\w_\bu : \bu\in \Lambda\}$ for any subset $\Lambda $ of $\Sigma^*$ such that $\bv|k+1, \bv|k+2,\ldots\in \Lambda$ and $\bu|k+2, \bu|k+3,\ldots \in \Lambda $ but $\bu|k+1\notin \Lambda$, where $|\bu\wedge \bv|=k$.
Then
$$
\dimh E^\w \geq s_A,
$$
for almost all $\w$.
\end{lem}

Finally, we prove that for  almost all $\w$, the critical value $s_A$ gives the Hausdorff dimension of the non-autonomous affine set $E^w$.
\begin{proof}[Proof of Theorem ~\ref{asams}]
By Theorem ~\ref{dimHU}, the critical value $s_A$ is the upper bound to the Hausdorff dimension of $E^w$, that is,  $\dimh E^\w\leq s_A$. We only need to show that $s_A$ is also the lower bound.

Fix $\bu,\bv\in \Sigma^*$. Let $n=|\bu \wedge \bv|$. Then
\begin{eqnarray*}
\Big|\Pi^\w(\bu)-\Pi^\w(\bv)\Big|&=& T_{\bu\wedge\bv}\Big(\w_{\bu|{n+1}}+ T_{n+1,u_{n+1}}\w_{\bu|{n+2}} +T_{n+1,u_{n+1}}T_{n+2,u_{n+2}} \w_{\bu|{n+3}} +\ldots\\
&&-(\w_{\bv|{n+1}}+T_{n+1,v_{n+1}}\w_{\bv|{n+2}}
+T_{n+1,v_{n+1}}T_{n+2,v_{n+2}}\w_{\bv|{n+3}} +\ldots)\Big)\\
&=&\Big|T_{\bu\wedge\bv}\Big(w_{\bu|{n+1}}+ q_n(\bu,\bv,\w)\Big)\Big|,
\end{eqnarray*}
where $q_n(\bu,\bv,\w)$ is a random vector which is independent of $\w_{\bu|{n+1}}$. Since the measure $P$ is absolutely continuous with respect to $\mathcal{L}^d$ with bounded density, we have that
\begin{eqnarray*}
&&\mathbb{E}\Big(|\Pi^\w(\bu)-\Pi^\w(\bv)|^{-s} \ \Big| \ \mathcal{F}\Big)=\int_{\D} \frac{dP(\w_{\bu|n+1})}{|\Pi^\w(\bu)-\Pi^\w(\bv)|^{s}}\\
&&=\int_0^\infty s \rho^{-s-1}P\Big\{\w_{\bu|n+1} \in \mathcal{D}: \big|\Pi^\w(\bu)-\Pi^\w(\bv)\big|< \rho \Big\} d\rho \\
&&=s \int_0^\infty \rho^{-s-1} P\Big\{\w_{\bu|n+1}\in  \mathcal{D} : \Big|T_{\bu\wedge\bv}\big(\w_{\bu|n+1}+ q_n(\bu,\bv,\w)\big)\Big|<\rho \Big\}d\rho\\
&&\leq C \int_0^\infty \rho^{-s-1} \mathcal{L}^d\Big\{\w_{\bu|n+1}\in  \mathcal{D} :\Big|T_{\bu\wedge\bv}\big(\w_{\bu|n+1}+ q_n(\bu,\bv,\w)\big)\Big|<\rho\Big\}d\rho.
\end{eqnarray*}

For simplicity, we write $\alpha_1,\ldots,\alpha_d$ for the singular values of $T_{\bu\wedge\bv}$. Let $m$ be the integer such
that $s<m\leq s+1$. Since $D$ is a bounded region in $\R^d$, there exists $\rho_0>0$ such that $D\subset B(0, \rho_0)$. It is clear that
$$
\{\w_{\bu|n+1}\in \mathcal{D} :|T_{\bu\wedge\bv}\big(\w_{\bu|n+1}+ q_n(\bu,\bv,\w)\big)|<\rho\}\subset
B(0,\rho_0)\cap \Big(T_{\bu\wedge\bv}^{-1}\big(B(0,\rho) \big) -  q_n(\bu,\bv,\w)\Big).
$$
By moving $B(0,\rho_0)$ to $ B(T_{\bu\wedge\bv}^{-1}(0),\rho_0) $, we have that
$$
\mathcal{L}^d\Big\{B(0,\rho_0)\cap \Big(T_{\bu\wedge\bv}^{-1}\big(B(0,\rho) \big) -  q_n(\bu,\bv,\w)\Big) \Big\}\leq
\mathcal{L}^d\big\{T_{\bu\wedge\bv}^{-1}\big(B(0,\rho)\big)\cap B\big(T_{\bu\wedge\bv}^{-1}(0),\rho_0\big)  \Big\}.
$$

Combining with these two facts,  we have that
\begin{eqnarray*}
&&\int_0^\infty \rho^{-s-1} \mathcal{L}^d\Big\{\w_{u_1\ldots u_{n+1}}\in  \mathcal{D} :\Big|T_{\bu\wedge\bv}\big(w_{u_1\ldots u_{n+1}}+ q_n(\bu,\bv,\w)\big)\Big|<\rho\Big\}d\rho \\
&\leq& \int_0^\infty \rho^{-s-1} \mathcal{L}^d\Big\{T_{\bu\wedge\bv}^{-1}\big(B(0,\rho)\big)\cap B\big(T_{\bu\wedge\bv}^{-1}(0),\rho_0\big)  \Big\}d\rho\\
&\leq& \int_0^{\alpha_m} \rho^{-s-1} \mathcal{L}^d \Big\{|x_1|\leq
\frac{\rho}{\alpha_1},\ldots,|x_m|\leq \frac{\rho}{\alpha_m},|x_{m+1}|\leq
\rho_0,\ldots,|x_n|\leq \rho_0
\Big\}dt \\
&& +\int_{\alpha_m}^\infty \rho^{-s-1} \mathcal{L}^d \Big\{|x_1|\leq
\frac{\rho}{\alpha_1},\ldots,|x_{m-1}|\leq
\frac{\rho}{\alpha_{m-1}},|x_m|\leq \rho_0,\ldots,|x_n|\leq \rho_0\Big\}dt  \\
&\leq& \int_0^{\alpha_m}
\frac{\rho_0^{n-m}\rho^{m-s-1}}{\alpha_1\ldots\alpha_m}d\rho
+\int_{\alpha_m}^\infty
\frac{\rho_0^{n-m+1}\rho^{m-s-2}}{\alpha_1\ldots\alpha_{m-1}}d\rho\\
&=&c_1\frac{1}{\alpha_1\ldots\alpha_{m-1}\alpha_m^{s-m+1}}+c_2\frac{1}{\alpha_1\ldots\alpha_{m-1}\alpha_m^{s-m+1}}\\
&=&\frac{c}{\phi^{s}(T_{\bu \wedge \bv})},
\end{eqnarray*}
where $c_1, c_2$ and $c$ are  constants independent of $\bu, \bv$ and $\w$.

Therefore, for all $0 < s\leq  d$ with $s$ non-integral,
we obtain that
$$
\mathbb{E}\Big(|\Pi^\w(\bu)-\Pi^\w(\bv)|^{-s}\ \Big| \ \mathcal{F}\Big) \leq  \frac{c}{\phi^{s}(T_{\bu \wedge \bv})},
$$
for all $\bu,\bv\in \Sigma^*$, $\bu \neq \bv$.
By Lemma ~\ref{thmHD},  we have that $\dimh\geq s_A$, and  the conclusion holds.

We next prove part (2). Choose $s$ such that $d<s<s_A$. By the definition of $s_A$, we have that $\mathcal{M}^s(\Sigma^\infty)=\infty$. By Lemma ~\ref{measure_ineq}, there exists a
Borel measure $\mu$ on $\Sigma^\infty$ such that
$0<\mu(\Sigma^\infty)<\infty$ and a constant $c_1>0$ such that
\begin{equation}
\mu(\mathcal{C}_\bu) \leq c_1 \phi^s(T_\bu),             \label{frost}
\end{equation}
for all $\bu \in \Sigma^*.  $
Let $\mu^\w$ be the projection measure of $\mu$ given by \eqref{def_ua}. It is clear that
$$\mu^\w(E^\w)=\mu(\Sigma^\infty)>0.$$

By similar argument as above, for all distinct $\bu,\bv\in \Sigma^\infty$ and $\rho>0$, there exist a constant $c$ such that
\begin{eqnarray*}
&&\mathbf{P}\{\w \in \Omega: |\Pi^\w(\bu)-\Pi^\w(\bv)|< \rho\}   \nonumber\\
&\leq& c \mathcal{L}^d \{T_{\bu\wedge\bv}^{-1}(B(0,\rho))\}    \nonumber\\
&\leq& c \mathcal{L}^d \Big\{(x_1,\ldots,x_d)\in\R^d: |x_1|\leq \frac{\rho}{\alpha_1},\ldots,|x_d|\leq \frac{\rho}{\alpha_d}\Big\}      \nonumber\\
&=& c\cdot \frac{\rho^d}{\phi^d(T_{\bu\wedge\bv})}\label{ineq_pw}.
\end{eqnarray*}
Furthermore, by \eqref{frost} and ~\eqref{alhpha+-}, we have that for $\bu\in\Sigma^k$,
\begin{equation*}\label{ineq_phid}
\frac{\mu(\mathcal{C}_\bu)}{\phi^d(T_\bu)}\leq c_1 \alpha_+^{(s-d)k}.
\end{equation*}

Similar to the proof of Theorem~\ref{dimHL}, we have that
\begin{eqnarray*}
&&\iint \liminf_{r\to 0} \frac{\mu^\w(B(x,r))}{ r^d}d\mu^\w(x) d\mathbf{P}(\w)  \\
&\leq& \liminf_{r\to 0} \frac{1}{ r^d} \iint \mathbf{P}\{\w \in \Omega:  |\Pi^\w(\bu)-\Pi^\w(\bv)|< r \} d\mu(\bu)d\mu(\bv)\\
&\leq& C \iint \frac{1}{\phi^d(T_{\bu\wedge\bv})}d\mu(\bu)d\mu(\bv)\\
&\leq& C \mu(\Sigma^\infty) \sum_{k=0}^\infty \alpha_+^{(s-d)k} \\
&<&\infty,
\end{eqnarray*}
where $C$ is a constant.

Hence for $\mathbf{P}$-almost all $\w$, we obtain that
$$\liminf_{r\to 0} \frac{\mu^\w(B(x,r))}{ r^d}<\infty,$$
for $\mu^\w$-almost all $x$.  Since $\mu^\w(E^\w)>0$, by Theorem~\ref{abcon},  it follows that $\mathcal{L}^d(E^\w)>0$ for $\mu^\w$-almost all $x$, and the conclusion holds.
\end{proof}

\section{Comparison of critical values and some examples}\label{sec_CV}
In this section, we discuss the relations of the critical values $s^*$, $s_A$ and the affine dimension  $d(T_1,\ldots, T_M)$.  The first conclusion follows straightforward from their definitions.

\begin{prop}\label{prop_CV}
Let $s^*$ and $s_A$ be given by \eqref{defs*} and \eqref{defsA}, respectively. Then
$$s_A\leq s^*.$$
\end{prop}
Note that the inequality may hold strictly, see Example~\ref{exCVie}. Moreover, $d(T_1,\ldots, T_M)$ generally does not exist in non-autonomous affine IFS.

In the following special cases, the two critical values $s_A$ and $s^*$ may coincide with  $d(T_1,\ldots, T_M)$.  Remind  that the non-autonomous affine set $E$ may not be self-affine fractals even if the sequence $\{\Xi_k\}$ is identical,  see Remark (4) after the definition of non-autonomous affine sets.

\begin{thm}\label{sequal}
Suppose that  $\Xi_k =\{T_1,T_2,\ldots,T_M\}$ and  $n_k=M$ for all $k>0$. Then
$$
s_A=s^*=d(T_1,\ldots,T_M).
$$
\end{thm}

\begin{proof}
Since $n_k =M$, we have that $\Sigma^k=\{1,\ldots, M\}^k$ for all $k>0$. This implies that
\begin{eqnarray*}
\sum_{\bu\in\Sigma^{k+l}}\phi^s(T_\bu)&=& \sum_{\bu\in\Sigma^k}\sum_{\bv\in\Sigma^l}\phi^s(T_{\bu\bv})\\
&\leq& \sum_{\bu\in\Sigma^k}\phi^s(T_{\bu}) \sum_{\bv\in\Sigma^l}\phi^s(T_{\bv}).
\end{eqnarray*}
Thus $\sum_{\bu\in\Sigma^k}\phi^s(T_\bu)$ is a submultiplicative sequence, so by the standard property of such sequences, $\lim_{k\to\infty}\big(\sum_{\bu\in\Sigma^k} \phi^s(T_\bu)\big)^\frac{1}{k}$ exists for each $s$. Since for each $\bu\in\Sigma^*$ and $0<h<1$,
$$
\phi^s(T_\bu) \alpha_-^h \leq \phi^{s+h}(T_\bu) \leq \phi^s(T_\bu) \alpha_+^h,
$$
we have that $\lim_{k\to\infty}\big(\sum_{\bu\in\Sigma^k} \phi^s(T_\bu)\big)^\frac{1}{k}$ is continuous and strictly decreasing in $s$. Since the limit is  greater
than 1 for $s = 0$ and less than 1 for sufficiently large $s$, there exists a unique $s$, written as $d(T_1,\ldots,T_M)=s$, such that
$$
\lim_{k\to\infty}\Big(\sum_{\bu\in\Sigma^k} \phi^{s}(T_\bu)\Big)^\frac{1}{k}=1.
$$

First, we show that $s^*\geq d(T_1,\ldots,T_M)$. For each $s>d(T_1,\ldots,T_M)$,
$$
\sum_{\bu\in\Sigma^*} \phi^s(T_\bu)= \sum_{k=1}^\infty \sum_{\bu\in\Sigma^k} \phi^s(T_\bu) <\infty.
$$
Hence
$$
\limsup_{\epsilon\to 0}\sum_{\bu\in\Sigma^*(s, \epsilon)}\phi^s(T_\bu)\leq \sum_{\bu\in\Sigma^*} \phi^s(T_\bu)<\infty.
$$
It follows that $s>s^*$. Then we have $s^*\leq d(T_1,\ldots,T_M)$.

Next, we show $d(T_1,\ldots,T_M)\leq s_A$. Let $m$ be the integer that $m-1\leq s_A<m$.
Arbitrarily choosing $s_A<s<m$, then $\mathcal{M}^s(\Sigma^\infty)=0$.
Thus there exists a covering set $\Lambda$ of $\Sigma^\infty$ such that
$$
\sum_{\bu\in \Lambda}\phi^s(T_\bu)\leq 1.
$$
Let $p=\max\{|\bu|:\bu\in \Lambda\}$. For $\epsilon<\alpha_-^p$, we define further covering sets $\Lambda_k$ by
$$
\Lambda_k=\{\bu_1,\bu_2,\ldots,\bu_q:\bu_i\in \Lambda, |\bu_1,\bu_2,\ldots,\bu_q|\geq k \text{ but } |\bu_1,\bu_2,\ldots,\bu_{q-1}| < k\}.
$$
Then by the submultiplicativity of $\phi^s$,
\begin{eqnarray*}
\sum_{\bu\in \Lambda}\phi^s(T_{\bu_1}\ldots T_{\bu_q}T_\bu) &\leq&  \phi^s(T_{\bu_1}\ldots T_{\bu_q}) \sum_{\bu\in \Lambda} \phi^s(T_\bu)\\
&\leq& \phi^s(T_{\bu_1}\ldots T_{\bu_q}).
\end{eqnarray*}
Applying this inductively, we obtain that
$$
\sum_{\bu\in \Lambda_k}\phi^s(T_\bu)\leq 1.
$$
If $\bu\in \Sigma^{k+p}$, then $\bu=\bu'\bv$, where $\bu'\in \Lambda_k$ and $|\bv|\leq p$. Moreover, for such each $\bu'$, there are at most $M^p$ such $\bv$. Since $\phi^s(T_\bu)\leq \phi^s(T_{\bu'})$,
$$
\sum_{\bu\in \Sigma^{k+p}} \phi^s(T_\bu) \leq M^p \sum_{\bu\in \Lambda_k}\phi^s(T_\bu) \leq M^p.
$$
Since it holds for all $k$,  we have that
$$\lim_{k\to \infty} \Big(\sum_{\bu\in \Sigma^k} \phi^s(T_\bu)\Big)^\frac{1}{k}\leq 1.
$$
Hence $s\geq d(T_1,\ldots,T_M)$, and it implies that  $d(T_1,\ldots,T_M)\leq s_A$. By Proposition~\ref{prop_CV},  we obtain that
$$
s_A=s^*=d(T_1,\ldots,T_M).
$$
\end{proof}
The Hausdorff dimension of self-affine sets immediately follows from Theorem~\ref{dimHL} and Theorem~\ref{sequal}, see\cite{Falco88, Solom98}
\begin{cor}
Let $E^\ba$ be the self-affine set given by ~\eqref{saattractor}, where $\ba=(a_1,\ldots,a_M)$.  Suppose that $\|T_j\|< \frac{1}{2}$ for all $0<j\leq M$.
Then for $\mathcal{L}^{M d}$-almost all $\ba\in \R^{M d}$,

$(1)$ $\dimh E^\ba = s_A$ if $s_A \leq d$,

$(2)$ $\mathcal{L}^d(E^\ba)>0$ if $s_A>d$.
\end{cor}
The Hausdorff dimension of self-affine sets with random translation immediately follows from Theorem~\ref{asams} and Theorem~\ref{sequal}, see\cite{JPS07}

\begin{cor}
Let $E^\w$ be the self-affine set with random translation. Then for $\mathbf{P}$-almost all $\omega$,

$(1)$ $\dimh E^\w = s_A$ if $s_A \leq d$,

$(2)$ $\mathcal{L}^d(E^\omega)>0$ if $s_A>d$.

\end{cor}

Finally, we give some examples to illustrate the definition of non-autonomous affine sets and our conclusions.
\begin{exmp}\label{exp1}
Suppose that $d=2$ and $J=[0,1]^2$. For all $k>0$, $n_k=2$,
$$
\Psi_{k,1}=\begin{pmatrix}
           \frac{1}{2} & 0 \\
           0 & \frac{2^{k+1}+1}{2^{k+1}+2}
         \end{pmatrix}x, \quad \Psi_{k,2}=\begin{pmatrix}
           \frac{1}{2} & 0 \\
           0 & \frac{2^{k+1}+1}{2^{k+1}+2}
         \end{pmatrix}x + \begin{pmatrix}
                            \frac{1}{2} \\
                            0
                          \end{pmatrix}.
$$
Then
$$
\lim_{k\to \infty} \max_{\bu\in \Sigma^k} |J_\bu| = \frac{2}{3},
$$
and $E$ has non-empty interior. Hence all dimensions of $E$ equals $2$.

\begin{figure}[h]\label{fig3}
\centering
\includegraphics[width=\textwidth]{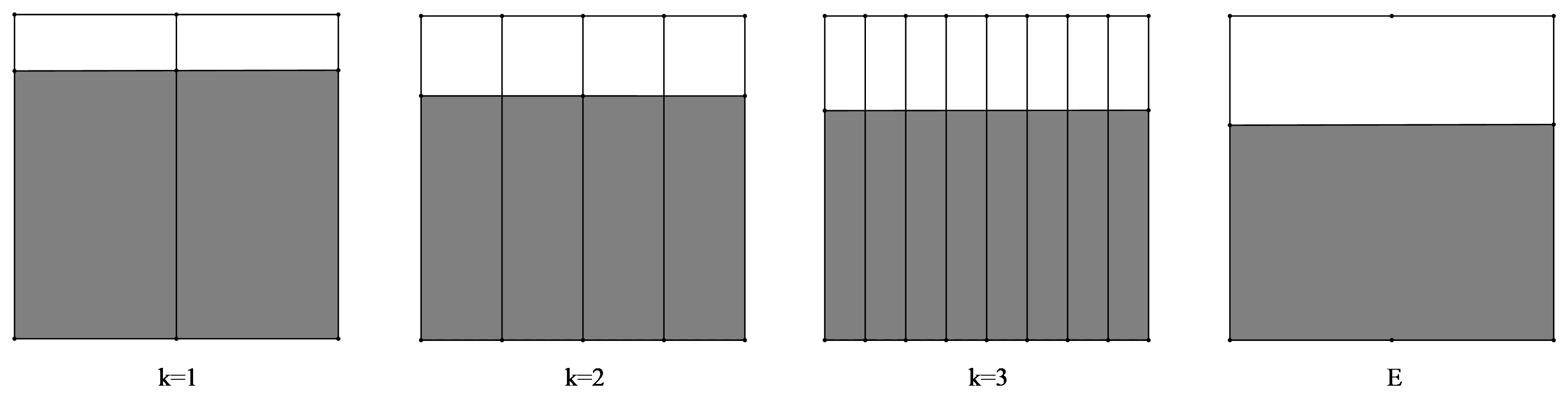}
\caption{non-autonomous affine set with positive finite Lebesgue measure}
\end{figure}
\end{exmp}

Note that Moran sets are always uncountable. In the following examples, we show that non-autonomous affine sets may be finite, countable or uncountable even if the Moran separation condition is satisfied.

\begin{exmp}\label{SMfnt}
Suppose that $d=2$ and $J=[0,1]^2$. For all $k>0$, $n_k=2$,
$$
\Psi_{k,1}=\begin{pmatrix}
           \frac{1}{2} & 0 \\
           0 & \frac{1}{2}
         \end{pmatrix}x, \quad
\Psi_{k,2}=\begin{pmatrix}
           0 & 0 \\
           0 & \frac{1}{2}
         \end{pmatrix}x+\begin{pmatrix}
         1 \\
         0
         \end{pmatrix}.
$$
Then
$$E=\{(0,0),(1,0), (\frac{1}{2},0), (\frac{1}{4},0), \ldots, (\frac{1}{2^n},0), \ldots\}.
$$
It is clear that $E$ is countable, see Figure~\ref{fig2}.

\begin{figure}[h]\label{fig2}
\centering
\includegraphics[width=\textwidth]{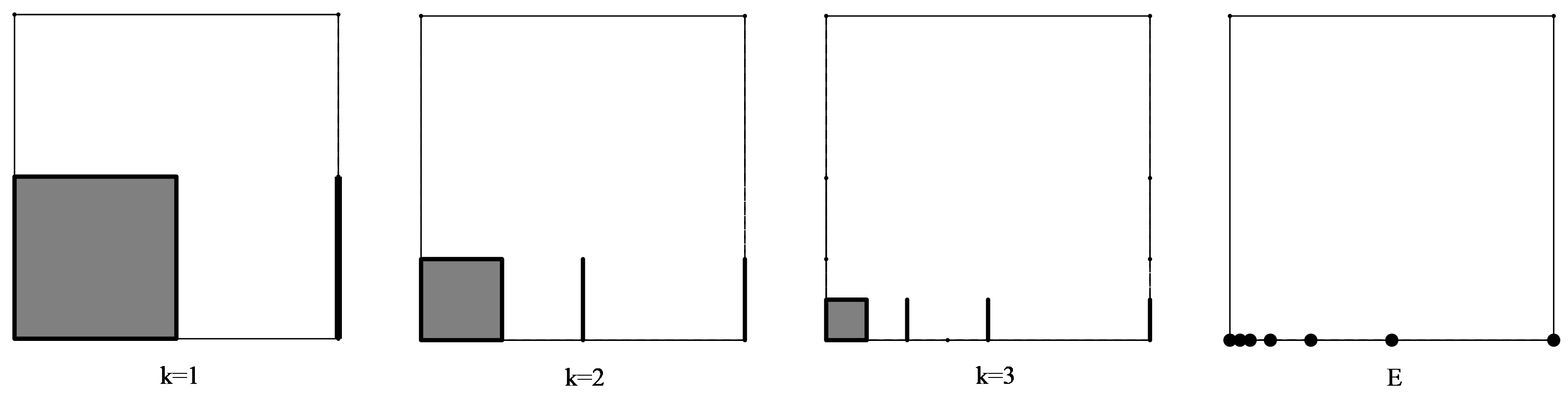}
\caption{non-autonomous affine set contains countably many elements}
\end{figure}
\end{exmp}

\begin{exmp}\label{SMcntb}
Given $J=[0,1]^2$ and an integer $n\geq 2$. Let  $n_1=n$  and  $\mathcal{D}\subset \{0,1,\ldots, n-1\}^2$. We write
$$
\Psi_{1,j}=\begin{pmatrix}
           \frac{1}{n} & 0 \\
           0 & \frac{1}{n}
         \end{pmatrix} (x+b_j),  \quad  b_j \in \mathcal {D}.
$$
For all $k\geq 2$, we set $n_k=2$ and
$$\Psi_{k,1}=\begin{pmatrix}
           \frac{1}{2} & \frac{1}{2} \\
           0 & \frac{1}{2}
         \end{pmatrix} x, \quad
\Psi_{k,2}=\begin{pmatrix}
           \frac{1}{2} & 0 \\
           \frac{1}{2} & \frac{1}{2}
         \end{pmatrix} x.
$$
Then $\card\ E=n$. See Figure~\ref{fig1}.

\begin{figure}[h]\label{fig1}
\centering
\includegraphics[width=\textwidth]{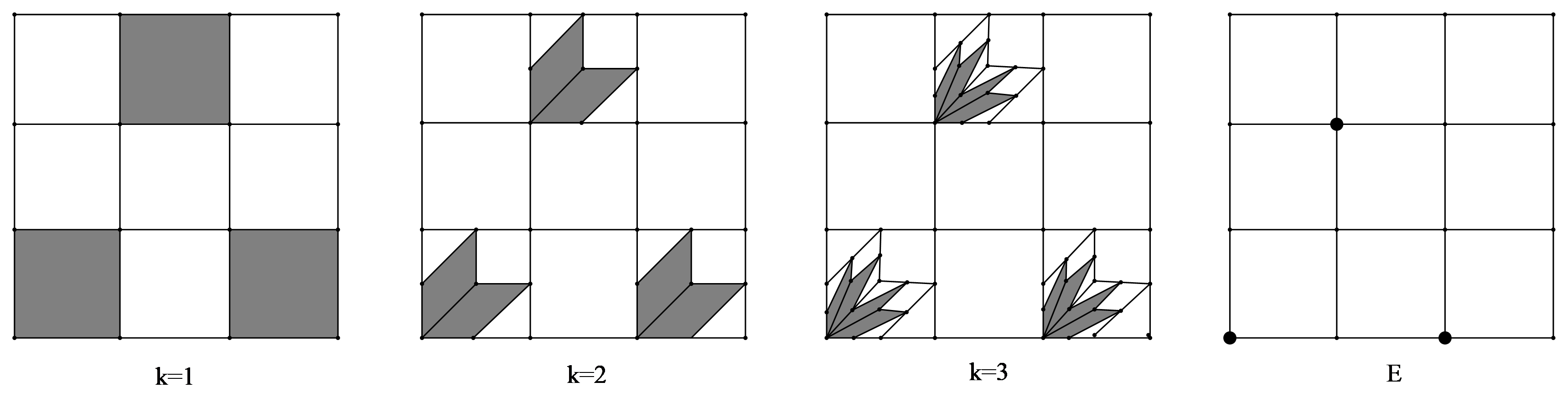}
\caption{Self-affine set contains only 3 elements, where $n=3$ and $\mathcal{D}=\{\binom{0}{0},   \binom{ 1 }{ 2 }, \binom{ 2 }{ 0 }   \}$. }
\end{figure}
\end{exmp}

Next, we give an example that $s^*$ and $s_A$ are the sharp bounds for the upper box dimension and the  Hausdorff dimension of non-autonomous affine sets.
\begin{exmp}\label{exCVie}
Suppose that $d=2$ and $J=[0,1]^2$. For each integer $k>0$,
$$
n_k=\left\{\begin{array}{cl}
                       3, & \qquad \textit{ if $k=1$ or } 3 \cdot 2^{2n}<k\leq 3\cdot 2^{2n+1} \textit{ for some integer } n\geq0\\
                       9, & \qquad \textit{ if } 3\cdot 2^{2n-1}<k\leq 3\cdot 2^{2n} \textit{ for some integer } n\geq 0.
                     \end{array}
\right.
$$
Let $\D_k\in \{0,1,\ldots,8\}\times \{0,1,2\}$ and $|\D_k|=n_k$. For $1\leq j\leq n_k$, $T_{k,j}=\diag\{\frac{1}{9},\frac{1}{3}\}$, and
$$
\Psi_{k,j}(x)=T_{k,j}(x+b_{k,j}),
$$
where $b_{k,j}\in\D_k$. Let $s_k$ be the number of $j$ such that $(i,j)\in\D_k$ for some $i$. Suppose that for each $k>0$, $s_k=2$.
Then for $\bu\in\Sigma^k$,
$$
\phi^s(T_\bu)=\left\{
  \begin{array}{lcc}
  2^{-sk}, & & 0<s<1;\\
  2^{(1-2s)k}, & & 1\leq s\leq 2.
  \end{array} \right.
$$
By simple calculation, we have that
$$
s^*=\frac{4}{3},\qquad s_A=\frac{7}{6}.
$$
But
$$
\overline{\dim}_{\rm B}  E = \frac{5\log 3+3\log 2}{6\log 3}<s^*,
$$
and
$$
\dimh E\leq \underline{\dim}_{\rm B} E = \frac{5\log 3+3\log 2}{6\log 3}<s_A.
$$
\end{exmp}

\end{document}